\documentclass[reqno]{amsart}
\usepackage[b5j,textheight=20cm,left=3cm,right=3cm]{geometry}
\usepackage[utf8]{inputenc}
\usepackage{graphicx,indentfirst}
\usepackage{mathrsfs,amsfonts,amssymb}
\usepackage{stmaryrd,cite}
\usepackage{amsmath}
\usepackage{cases}
\usepackage{enumerate}
\usepackage[usenames,dvipsnames,svgnames]{xcolor}
\definecolor{refkey}{rgb}{1,0,0.5}
\definecolor{labelkey}{rgb}{0,0.4,1}
\usepackage{todonotes}
\usepackage{marginnote}

\presetkeys{todonotes}{fancyline, color=green!40}{}

\makeatletter
\renewcommand{\@todonotes@drawMarginNoteWithLine}{%
	\begin{tikzpicture}[remember picture, overlay, baseline=-0.75ex]%
	\node [coordinate] (inText) {};%
	\end{tikzpicture}%
	\marginnote[{
		\@todonotes@drawMarginNote%
		\@todonotes@drawLineToLeftMargin%
	}]{
		\@todonotes@drawMarginNote%
		\@todonotes@drawLineToRightMargin%
	}%
}
\makeatother

\usepackage{listings}
\usepackage{ifpdf}
\ifpdf
\usepackage[CJKbookmarks=true,
         hyperindex=true,
         pdfstartview=FitH,
         bookmarksnumbered=true,
         bookmarksopen=true,
         colorlinks=true,
         citecolor=blue,
         linkcolor=blue,
         urlcolor=blue,
         pdfborder=001,
         pdfauthor={Chengchun Hao},
         pdftitle={},
         pdfkeywords={},
         ]{hyperref}
\else
\usepackage[
         hypertex,
         hyperindex,
         linkcolor=blue,
         unicode,
         citecolor=blue%
         ]{hyperref}
\fi
\usepackage{cleveref}
\allowdisplaybreaks
\numberwithin{equation}{section}
\usepackage{amsthm}

\theoremstyle{definition}

\theoremstyle{remark}

\numberwithin{equation}{section}
\newtheorem{thm}{Theorem}[section]

\newtheorem{prop}[thm]{Proposition}
\newtheorem{rmk}[thm]{Remark}

\usepackage{color}

\newcommand{\be}{\begin{equation}}
\newcommand{\bs}{\begin{split}}
\newcommand{\es}{\end{split}}
\newcommand{\ee}{\end{equation}}
\newcommand{\bee}{\begin{equation*}}
\newcommand{\eee}{\end{equation*}}

\allowdisplaybreaks

\begin{document}

\author{Qiangchang Ju$^{\dag}$}\thanks{$^{\dag}$
              Institute of Applied Physics and Computational Mathematics, Beijing 100088, P.R.China.
              Email: {ju\_qiangchang@iapcm.ac.cn}}\

\author{Tao Luo$^{\ddag}$}\thanks{$^{\ddag}$Department of Mathematics, City University of Hong Kong, 83 Tat Chee Avenue, Kowloon Tong, Hong Kong.
E-mail: taoluo@cityu.edu.hk}

\author{Xin Xu $^{\S}$}\thanks{$^{\S}$
School of Mathematical Sciences and College of Oceanic and Atmospheric Sciences,
Ocean University of China, Qingdao, P.R.China.
         Email:   xx@ouc.edu.cn}

\title[]{\bf  Singular Limits for the Navier-Stokes-Poisson Equations of Viscous Plasma with Strong Density Boundary Layer}
\begin{abstract}
The quasi-neutral limit of the Navier-Stokes-Poisson system modeling a viscous plasma with vanishing
viscosity coefficients in the half-space $\mathbb{R}^{3}_{+}$ is rigorously proved under a Navier-slip boundary condition for velocity and the Dirichlet boundary condition for electric potential. This is achieved by establishing the nonlinear stability of the approximation solutions involving the strong boundary layer in density and electric potential, which comes from the break-down of the quasi-neutrality near the boundary, and dealing with the difficulty of the interaction of this strong boundary layer
with the weak boundary layer of the velocity field. 
\end{abstract}

\maketitle

\section{Introduction}
\label{intro}
In this paper,  we investigate the singular behavior of solutions to a hydrodynamic model of a viscous plasma  in a three dimensional domain with a physical boundary.  The model we consider here is for the behavior of ions in a background of massless electrons. Under the massless assumption,  the electrons follow the classical Maxwell-Boltzmann relation:  let $\rho_e$ be their density, then $\rho_e=e^{-\phi}$ after suitable normalization
of constants, where $\phi$ is the electric potential. The plasma considered in this paper is unmagnetized, consisting of free electrons and a single species of ions that form a compressible viscous fluid, the motion of which is governed by the Navier-Stokes system. We consider this system in the three dimensional half-space:
For time-space variable $(t,x)=(t,x_1,x_2,x_3)=(t,y,x_3)\in \mathbb{R}_+\times \mathbb{R}^2\times \mathbb{R}_+:=\mathbb{R}_+\times\Omega$,
\begin{equation}   \label{1.1}
\left\{
\begin{array}{lll}
\rho_t+\nabla\cdot(\rho u)=0,\\
(\rho u)_t+\nabla\cdot(\rho u\otimes u+T^i\rho\mathbb{I})=\rho\nabla\phi+\mu'\triangle u+(\mu'+\nu')\nabla\nabla\cdot u,\\
\lambda\triangle\phi+e^{-\phi}=\rho,
\end{array}
\right.
\end{equation}
where the unknowns $\rho$ and $ u$ are the density  and velocity of ions, respectively. The average temperature of the ions is denoted by $T^i$ and the squared scaled Debye length is denoted by $\lambda$.  The parameters $\mu'$ and $\nu'$ are viscosity coefficients with $\mu'>0$ and $\mu'+\nu'>0$.

On the boundary $\partial\Omega$, the Navier-slip boundary conditions are imposed on the velocity field, i.e.
$u\cdot n=0, \ (Su\cdot n)_{\tau} +\beta u_{\tau}=0$, where $\beta$ is a positive
constant  measuring the tendency of the fluid to slip on the boundary, $n$ is the unit outer normal to $\partial \Omega$,
$S$ is the strain tensor,
$$Su=\frac{1}{2} ( \nabla u+\nabla u^t)$$
and for some vector field $w$ on $\partial \Omega$, $w_\tau$ stands for the tangential part of $w$, i.e.,
$w_\tau=w-(w\cdot n) n.$ For the simplicity of the presentation, we take $\beta=\frac{1}{2}$. In this case, for
$\Omega=\mathbb{R}^{3+}$, the boundary conditions for the velocity field $u$ read that

\begin{align}  \label{1.2}
u_3=0, \qquad u_i-\frac{\partial u_i}{\partial x_3}=0, \qquad i=1,2.
\end{align}
The Dirichlet boundary condition is imposed  on the electric potential
\begin{align}  \label{1.2'}
\phi=\phi_b(y),
\end{align}
i.e. the electric potential is prescribed on $\partial \Omega$. We consider the case
that $\phi_b(y)$ is smooth and compactly supported, without loss of generality.

The aim of this paper is to study the asymptotic behavior of smooth solutions to  system \eqref{1.1} with boundary condition \eqref{1.2}-\eqref{1.2'} in the regime of small Debye length and small viscosity. For this purpose, we assume that
\begin{align}  \label{1.3}
\mu'=\mu\epsilon^2,\qquad \nu'=\nu\epsilon^2,\qquad \lambda=\epsilon^2.
\end{align}
Formally, it yields by letting $\epsilon=0$ in \eqref{1.1} that
\begin{equation}   \label{1.4}       
\left\{
\begin{array}{lll}
\rho_t+\nabla\cdot(\rho u)=0,\\
(\rho u)_t+\nabla\cdot(\rho u\otimes u+T^i\rho\mathbb{I})=\rho\nabla\phi,\\
e^{-\phi}=\rho,
\end{array}
\right.
\end{equation}
which can be rewritten as the following compressible Euler system
\begin{equation}   \label{euler}       
\left\{
\begin{array}{lll}
\rho_t+\nabla\cdot(\rho u)=0,\\
u_t+u\cdot\nabla u+(T^i+1)\nabla\ln\rho=0.
\end{array}
\right.
\end{equation}
For this system, only the non-penetration boundary condition $u_3|_{x_3=0}=0$ is required. Hence, there is a loss of boundary conditions
which leads to the appearance of the boundary layers when $\epsilon$ tends to zero.

The most interesting and difficult part of the singular limit problem studied in this paper is to deal with the interaction of the strong (of amplitude $O(1)$) boundary layer in the density $\rho$ and the weak (of amplitude $O(\epsilon)$) boundary layer in the velocity field $u$,   as $\epsilon\to 0$. The presence of the strong boundary layer in $\rho$ is due to the Dirichlet boundary condition \eqref{1.2'} in the electric potential $\phi$ since the solution of \eqref{euler} cannot
 in general satisfy $\rho|_{x_3=0}=e^{-\phi_b}$. Indeed, for the Neumann boundary condition for the electric potential, $\nabla_{n} \phi=0$ on $\partial \Omega$,  considered in literatures, for instance, \cite{DFN15, FZ-ARMA },  only weak boundary layers in the density $\rho$ appears. In fact, for the limiting problem \eqref{euler} with the non-penetration boundary condition $u_3|_{x_3=0}=0$, on the boundary $x_3=0$, since $u_3=0$, one has $\frac{(T^i+1)}{\rho}\partial_{x_3}\rho=-(\partial_t+u_y\cdot\nabla_y) u_3=0$ ($u_y=(u_1, u_2)$, $\nabla_y=(\partial_{x_1}, \partial_{x_2})$), therefore ,
 $\partial_{x_3}\rho=0$,
and thus $\partial_{x_3}\phi=-\frac{1}{\rho}\partial_{x_3}  {\rho}=0$ on the boundary $x_3=0$, which matches the Neumann boundary condition for $\phi$ for the original problem. Hence,
only weak boundary layers in $\rho$ and $\phi$ appear if one replaces the Dirichlet boundary condition \eqref{1.2'} by  the Neumann boundary condition  $\partial_{x_3} \phi=0$ on $\partial \Omega$.
  This distinguishes the problem we consider here from other problems in fluids or plasma equations for which only weak boundary layers appear. The examples of this type of problems with weak boundary layers include  the inviscid limit problem of the Navier-Stokes equations with  the Navier-slip boundary condition \eqref{1.2}, and the combined vanishing viscosity limits under the Navier-slip boundary conditions and quasi-neutral limits of the Navier-Stokes-Poisson system of plasma for which the electronic potential satisfies the Neumann boundary condition instead of
 the Dirichlet boundary condition which we study in this paper. The inviscid limit problem of the Navier-Stokes equations of fluids with
 the Navier-slip boundary condition  has been extensively studied for both compressible and incompressible flows by various approaches (\cite{B-C-2010, {C-M-R-1998}, I-P-2006, I-S-2011, K-2006, M-R-2012, W-M-2012, W-X-Y-2015, X-X-2007}).
Indeed,  it is shown in \cite{W-M-2012,W-X-Y-2015}  that solutions to the compressible Navier-Stokes equations with Navier-slip boundary conditions have the following approximations
\begin{align}
\rho^{NS}=\rho^{Euler}(t,x)+\epsilon^{3}\Upsilon(t,y,\frac{x_3}{\epsilon})+O(\epsilon^4),\label{1.7}\\
u^{NS}=u^{Euler}(t,x)+\epsilon U(t,y,\frac{x_3}{\epsilon})+O(\epsilon^2),\label{1.8}
\end{align}
where $\Upsilon$ and $U$ are smooth profiles with fast decay in the last variable which indicate that the boundary layers for both velocity and density have smaller amplitudes, of the order of $O(\epsilon^{3})$ and $O(\epsilon)$, respectively.  Therefore, they appear as weak boundary layers.  Furthermore, the boundary layers for the density are weaker than the one for the velocity.

For a Navier-Stokes-Poisson system of plasma, Donatelli, Feireisl and Novotn\'y \cite{DFN15} studied a singular limit for the  Navier-Stokes-Poisson system in a bounded domain under the boundary conditions  that the velocity field $u$ satisfies the Navier-slip boundary conditions and the electric potential $\phi$ satisfies the Neumann boundary condition, excluding the strong boundary layers as shown in \cite{DFN15}. This is a key difference from the case of the Dirichlet boundary condition  for
$\phi$ we study in the present paper. Moreover, only $L^2$-convergence of weak solutions was discussed in
\cite{DFN15} without estimates on derivatives, and no boundary layer analysis was given.  In the present paper,
we prove that the approximate solutions involving boundary layer profiles are nonlinear stable with detailed regularity estimates, giving a clear picture of the singular behavior of solutions near the boundary for small Debye length and viscosity.

For the Navier-Stokes equations with the Navier-slip boundary condition, the uniform regularity of solutions which yields the vanishing viscosity limit by the
compactness  was established by Masmoudi and Rousset \cite{M-R-2012} for the incompressible flow by using the  conormal Sobolev estimates.  This approach was later adopted to study the compressible isentropic flow by Wang, Xin, and Yong \cite{W-X-Y-2015}(see also \cite{WangYong} for the nonisentropic flow).  The weaker boundary layer of the size $O(\epsilon^{3})$ in the density $\rho$ plays a crucial role in \cite{W-X-Y-2015} to extend the uniform regularity estimates for the incompressible flow in \cite{M-R-2012} to the compressible one. However,
 this approach of establishing the uniform estimates of the solutions in  conormal Sobolev space  is not applicable to the problem that we study in this paper due to the  strong boundary layer in the density. Indeed, as it will be shown  later, the boundary layer expansion for the density $\rho$ for system \eqref{1.1} with boundary condition \eqref{1.2}-\eqref{1.2'} takes the form of
\begin{align}
\rho^{NSP}=\rho^{Euler}(t,x)+\Upsilon^0(t,y,\frac{x_3}{\epsilon})+O(\epsilon),
\end{align}
and velocity has a similar expansion to \eqref{1.8}, where profile $\Upsilon^0$ is smooth and fast decreasing in the last variable. Due to the strong boundary layer in the density $\rho$ and its interaction with the weak boundary layer of the size of $O(\epsilon)$ of  the velocity field $u$, it poses a great challenge to rigorously justify the small viscosity and Debye length limit.

The strong density boundary layer was first studied by G\'{e}rard-Varet, Han-Kwan and Rousset in \cite{G-H-R-2013, G-H-R-2016}, where the quasi-neutral limits of the isothermal Euler-Poisson system with both subsonic and supersonic outflow boundary conditions were investigated. By constructing approximate solutions and then showing their stability, the authors proved rigorously the quasi-neutral limit of the system.  The approach  used in \cite{G-H-R-2013, G-H-R-2016} to prove the stability can be regarded as a sort of ``hyperbolic'' method, since the normal derivatives of the solutions can be represented by the tangential derivatives of the solutions for Euler-Poisson system and thus the normal derivative estimates can be obtained directly from the tangential derivative estimates and the vorticity estimates. However,  such hyperbolic method breaks down for our viscous model.  For the problem we study in this paper, the key issue is to deal with the nonlinear coupling of the strong density boundary layer and weak boundary layer of the velocity filed,  for which new ways to obtain the normal derivatives of the solution need to be developed, compared with the approach used in \cite{G-H-R-2013, G-H-R-2016}.  We will give a detailed description of our approach in the next section.

  An approximate solution, up to any order,  involving the boundary layer corrections to the problem \eqref{1.1}-\eqref{1.2'} was constructed recently by Ju and Xu  \cite{Ju-Xu}. Moreover,  the linearized stability of the approximation solution is justified in \cite{Ju-Xu}. However, the nonlinear stability was left open in \cite{Ju-Xu} due to serious difficulties caused by the strong boundary layer and the characteristic boundary which create a great challenge in the estimates of the interaction of the strong boundary layer of the density and the weak boundary layer of the velocity fields. Indeed, in the asymptotic expansion in \cite{Ju-Xu} near the boundary, the normal derivative of the density is   of the order of  $O(\epsilon^{-1})$ and the second normal derivative of the velocity field, $\partial^2_{x_3} u$, is of the size $O(\epsilon^{-1})$, i.e., the strong boundary layer appears in the normal
derivative of the velocity field.  Dealing with such extremely singular terms
is the most difficult part of the nonlinear analysis, compared with the linear
stability analysis in \cite{Ju-Xu}.  Therefore, it is non-trivial to extend
the linear stability to a nonlinear one.

It should be noted that we discuss the Navier-slip boundary condition for the velocity field in this paper, instead of the Dirichlet boundary condition $u=0$ on $\partial \Omega$ for which the boundary layer is a strong characteristic one described by the Prandtl equations, and the justification of vanishing viscosity limit is a
major open problem in the mathematical theory of fluid mechanics, for which results are available only for
some special cases such as the analytic data (\cite{ Sammartino,  NNT}),   the case when the vorticity is away from the boundary (\cite{ Maekawa,  FTZ }) , or the steady state Navier-Stokes equations (\cite{ GuoNgu,  GuoS1, GuoS }). For MHD flow with certain
boundary condition on the magnetic field, a cancelation of the leading singular terms is used in \cite{  LXY } to justify
the vanishing viscosity and magnetic diffusion limit. However, such a cancellation is not available for system \eqref{1.1} when the Dirichlet boundary condition $u|_{\partial \Omega} =0$ is imposed.

Before ending this introduction, we give some references related to this paper.  For the Navier-Stokes-Poisson system, the combined quasi-neutral and vanishing viscosity limits has  been justified for weak solutions by Wang and Jiang \cite{W-J} in the torus. The quasi-neutral limit of either Euler-Poisson system or Navier-Stokes-Poisson system with fixed viscosity coefficients has been intensively studied. For the related references, one may
refer to \cite{C-G-2000,FZ-ARMA, H-2011, J-L-L-2009, L-J-X-2016, P-W-2004, S-N-2001,W-2004}.  The problem of formation and dynamics of the
plasma sheath was investigated in  \cite{10, 20}  and
references therein. Based on a formal expansion, a two-fluid
quasi-neutral plasma was studied in \cite{9}. For the existence of weak solutions to the Navier–Stokes–Poisson equations, interested readers may refer to \cite{DL15,K-S2008,ZT07} and the references therein for more details.

Throughout this paper, the positive generic
constants that are independent of $\epsilon$ are denoted by $C$. We use $H^s$ to denote the usual Sobolev space, and the corresponding Sobolev norm is denoted by $\|\cdot\|_{H^s}$. Moreover, we denote $\|\cdot\|=\|\cdot\|_{L^2(\Omega)}$ for simplicity.
The notation $|\cdot|_{H^m(\partial\Omega)}$ will be used for the standard Sobolev norm of functions defined on the boundary $\partial\Omega$.
We also set $u_y=(u_1,u_2)$, $\nabla_y=(\partial_1,\partial_2)$, $\triangle_y=\partial^2_1+\partial^2_2$. Finally, the standard commutator of the operators $A$ and $B$ is denoted by $[A,B]=AB-BA$.


\section{Main results}
\label{sec:1}
\subsection{Approximation solutions}
First, we recall the following results concerned with the approximation solution. which was proved in \cite{Ju-Xu} .
\begin{prop} \label{thm2.1}
Let $m\geq3$, $K\in\mathbb{N}_+$, and $\tilde{\rho}(x)$ be a  smooth function with a positive lower bound. Assuming that $(\rho^0_0,u_0^0)$ satisfies compatibility conditions with the boundary conditions and $(\rho^0_0-\tilde{\rho},u_0^0)\in H^{m+2K+3}(\mathbb{R}^3_+)$, then there exists $T^{\ast}>0$ independent of $\epsilon$ and a smooth approximation solution $(\rho_a,u_a,\phi_a)$ of order $K$ to \eqref{1.1}-\eqref{1.2} of the form
\begin{align}
\rho_a=\sum_{i=0}^K\epsilon^i(\rho^i(t,x)+\Upsilon^i(t,y,\frac{x_3}{\epsilon})),\nonumber\\
u_a=\sum_{i=0}^K\epsilon^i(u^i(t,x)+U^i(t,y,\frac{x_3}{\epsilon})),\nonumber\\
\phi_a=\sum_{i=0}^K\epsilon^i(\phi^i(t,x)+\Phi^i(t,y,\frac{x_3}{\epsilon})),\nonumber
\end{align}
such that\\
\textrm{\textbf{(i)}} The leading order $(\rho^0,u^0)$ is the solution of isothermal Euler equation
\begin{equation}   \label{1.4}       
\left\{
\begin{array}{lll}
\rho^0_t+\nabla\cdot(\rho^0 u^0)=0,\\
\rho^0(u^0_t+u^0\cdot\nabla u^0)+(T^i+1)\nabla\rho^0=0.
\end{array}
\right.
\end{equation}
on $[0,T^{\ast}]$ with initial data $(\rho^0_0,u_0^0)$ and the non-penetration boundary condition $u^0_3=0$, such that
$$(\rho^0-\tilde{\rho},u^0)\in C^0([0,T^{\ast}],\ H^{m+2K+3}(\mathbb{R}^3_+)).$$
And $\phi^0$ is determined by the relation $e^{-\phi^0}=\rho^0$.\\
\textrm{\textbf{(ii)}} For any $1\leq j\leq K$, the higher-order terms $(\rho^j,u^j,\phi^j)$ satisfy
$$(\rho^j,\ u^j,\ \phi^j)\in C^0([0,T^{\ast}],\ H^{m+3}(\mathbb{R}^3_+)).$$
\textrm{\textbf{(iii)}} The smooth profiles $(\Upsilon^j,\ U^j,\ \Phi^j)$ and their derivatives are exponentially decay functions with respect to the fast variable $z=\frac{x_3}{\epsilon}$. In particular, the leading order term $U^0(t,x,z)\equiv0$.
\\
\textrm{\textbf{(iv)}} Let $(\rho^\epsilon,u^\epsilon,\phi^\epsilon)$ be a solution of \eqref{1.1} and define the error term $(\rho,u,\phi)$ as
$$\rho=\rho^\epsilon-\rho_a, \qquad u=u^\epsilon-u_a,\qquad \phi=\phi^\epsilon-\phi_a.$$
Then $(\rho,u,\phi)$ satisfies
\begin{equation}   \label{2.4}       
\left\{
\begin{array}{lll}
\partial_t\rho+(u_a+u)\cdot\nabla\rho+\rho\nabla\cdot(u+u_a)+\nabla\cdot(\rho_au)=\epsilon^KR_\rho,\\
\partial_t u+(u_a+u)\cdot\nabla u+u\cdot\nabla u_a+T^i(\frac{\nabla \rho}{\rho+\rho_a}-\frac{\nabla\rho_a}{\rho_a}(\frac{\rho}{\rho_a+\rho}))\\
\qquad\qquad=\nabla\phi+\frac{\mu\epsilon^2}{\rho_a+\rho}\triangle u+\frac{(\mu+\nu)\epsilon^2}{\rho_a+\rho}\nabla\nabla\cdot u+h(\rho, \rho_a, u_a)+\epsilon^KR_u,\\
\epsilon^2\triangle\phi=\rho-e^{-\phi_a}(e^{-\phi}-1)+\epsilon^{K+1}R_\phi,
\end{array}
\right.
\end{equation}
where the reminders $(R_\rho,R_u,R_\phi)$ satisfying
\begin{align}  \label{2.5}
\sup_{[0,T^{\ast}]}\|\nabla^\alpha(R_\rho,R_u,R_\phi)\|_{L^2(\mathbb{R}^3_+)}\leq C\epsilon^{-\alpha_3},\qquad \forall \alpha=(\alpha_1,\alpha_2,\alpha_3),
\end{align}
and $h(\rho, \rho_a, u_a)$ is given by
\begin{align}  \label{2.6}
h(\rho, \rho_a, u_a)=-\frac{\mu\epsilon^2\rho}{\rho_a(\rho_a+\rho)}\triangle u_a-\frac{(\mu+\nu)\epsilon^2\rho}{\rho_a(\rho_a+\rho)}\nabla\nabla\cdot u_a.
\end{align}
\end{prop}
\begin{rmk}
Generally, the leading order terms $\Upsilon^0$ and $\Phi^0$ are nonzero. This implies that the strong boundary layers for density and electric field will appear in the limit process.
\end{rmk}
\subsection{Main Theorem}
The aim of this paper is  to establish the nonlinear stability  of the approximation solution constructed in the above theorem. To this end, we complete the system \eqref{2.4} with the following initial and boundary conditions. The initial conditions are given by
\begin{align} \label{2.7}
\rho|_{t=0}=\epsilon^{K+1}\rho_0, \qquad u|_{t=0}=\epsilon^{K+1}u_0,
\end{align}
and the boundary conditions are
\begin{align} \label{2.8}
u_3|_{x_3=0}=0, \qquad (u_i-\frac{\partial u_i}{\partial x_3})|_{x_3=0}=0, \qquad i=1,2, \qquad \phi|_{x_3=0}=0.
\end{align}
The main results of this paper are stated as follows.
\begin{thm}  \label{thm3.1}
Under the assumptions of Proposition \ref{thm2.1}, assume further that the initial data \eqref{2.7} satisfies $(\rho_0,u_0)\in H^6(\mathbb{R}^3_+)$ and the compatibility conditions with the boundary conditions \eqref{2.8}. Then for $K>6$ and sufficiently small $\epsilon$, there exists $T>0$ independent of $\epsilon$ such that the initial boundary value problem \eqref{2.4}, \eqref{2.7} and \eqref{2.8} admits a unique solution $(\rho,u,\phi)$ on $[0,T]$ and
\begin{align}
\sup_{0\leq t\leq T}\|(\rho,u,\phi,\epsilon\nabla\phi)\|_{H^3(\mathbb{R}^3_+)}\leq C\epsilon^{K-5}.\label{es1}
\end{align}
\end{thm}

As a corollary of Theorem \ref{thm3.1}, using the estimates \eqref{es1} and Sobolev embedding inequalities, it yields the following results:
\begin{align}
\sup_{0\leq t\leq T}\|(\rho^\epsilon&-\rho^0-\Upsilon^0)\|_{L^\infty(\mathbb{R}^3_+)}\nonumber\\
&\leq \sup_{0\leq t\leq T}\big(\|(\rho^\epsilon-\rho_a)\|_{L^\infty(\mathbb{R}^3_+)}
+\|\sum_{i=1}^K\epsilon^i(\rho^i+\Upsilon^i)\|_{L^\infty(\mathbb{R}^3_+)}\big)\nonumber\\
&\leq C\epsilon,\nonumber
\end{align}
and
\begin{align}
\sup_{0\leq t\leq T}&\|(\rho^\epsilon-\rho^0)\|_{L^2(\mathbb{R}^3_+)}\nonumber\\
&\leq\sup_{0\leq t\leq T}\big(\|(\rho^\epsilon-\rho_a)\|_{L^2(\mathbb{R}^3_+)}
+\|\Upsilon^0\|_{L^2(\mathbb{R}^3_+)}+\|\sum_{i=1}^K\epsilon^i(\rho^i+\Upsilon^i)\|_{L^2(\mathbb{R}^3_+)}\big)\nonumber\\
&\leq C\epsilon^{\frac{1}{2}}.\nonumber
\end{align}
Similarly, we have
\begin{align}
\sup_{0\leq t\leq T}(\|(u^\epsilon-u^0)\|_{L^\infty(\mathbb{R}^3_+)}+\|(u^\epsilon-u^0)\|_{L^2(\mathbb{R}^3_+)})&\leq C\epsilon,\nonumber\\
\sup_{0\leq t\leq T}\|(\phi^\epsilon-\phi^0-\Phi^0)\|_{L^\infty(\mathbb{R}^3_+)}&\leq C\epsilon,\nonumber\\
\sup_{0\leq t\leq T}\|(\phi^\epsilon-\phi^0)\|_{L^2(\mathbb{R}^3_+)}&\leq C\epsilon^{\frac{1}{2}}.\nonumber
\end{align}

The proof of this theorem  is by a bootstrap argument based on the local well-posedness for  system \eqref{2.4}, \eqref{2.7} and \eqref{2.8} for fixed $\epsilon$.

The key of bootstrap argument  is to establish the high order estimates of
Sobolev norms, dealing with those extremely singular nonlinear terms which do not appear in the linear analysis in \cite{Ju-Xu}. The most difficult part is on the normal derivative estimates, high-order normal derivatives and mixed derivatives of the solutions, due to the strong boundary layer in $\rho$ and its interaction with the weak boundary layer of $u$. In each step, we have to identify the precise control of various norms of the error $(\rho, u, \phi)$ in terms of $\epsilon$ so that the bootstrap argument can be closed.  The strategy of the proof is following:
  First, we prove the basic $L^2$ estimate of the solution. Even this part is
quite new and difficult compared with that for linear stability;  one has to deal with some singular terms in the error system \eqref{2.4}, again, due to the interaction of the  strong boundary layer in the density and weak boundary layer in velocity.  A proper cancelation of the singular terms is identified to obtain the $L^2$ estimate. For the estimates of the high-order tangential derivatives, using the strategy of the $L^2$ estimate,   we apply the estimates of  conormal derivatives.  Here, the most complicated and elaborate parts are to deal with the commutator estimates because the conormal derivatives do not commute with the usual derivatives.

Based on the  obtained the estimates of $L^2$ and tangential derivatives, we
then focus on the most difficult estimates of the paper, the normal derivative estimates, high-order normal derivatives and mixed derivatives of the solutions. This part is highly non-trivial and the strategy is quite roundabout: we first derive the normal derivatives estimates of density $\rho$, then recover the estimates for normal derivatives of $u$ and $\phi$ by the equations. In the first part of this strategy of  the normal derivatives estimates of density $\rho$, there appears a highly singular term,  the crossing term $\partial_{x_3}\rho \partial^2_{x_3} u$.  To deal with this highly singular term,  an interesting cancelation is employed to obtain a transport equation \eqref{nor6} for the normal derivative of $\rho$ with the source terms which can be estimated by using the obtained tangential estimates. After we obtain the estimates of the normal derivative of $\rho$,  the estimates of the normal derivative of the vertical velocity and electric potential can be derived by the static estimates. Since the boundary is characteristic, we do not  have any information on the normal derivative of the horizontal velocity. Motivated by the work of N. Masmoudi, F. Rousset \cite{M-R-2012}, we estimate the vorticity for which the corresponding boundary conditions can be determined.  Finally, the estimates of high-order normal derivatives and mixed derivatives can be obtained step by step.

We mention that the analysis in the present paper can be extended to the corresponding isentropic model
\begin{equation}
\left\{
\begin{array}{lll}
\rho_t+\nabla\cdot(\rho u)=0,\\
(\rho u)_t+\nabla\cdot(\rho u\otimes u)+\nabla p=\rho\nabla\phi+\mu'\triangle u+(\mu'+\nu')\nabla\nabla\cdot u,\\
\lambda\triangle\phi+e^{-\phi}=\rho,
\end{array}
\right.\nonumber
\end{equation}
where the pressure $p=p(\rho)$ is a smooth function of $\rho$ with $p'(\rho)>0$ for $\rho>0$ and the limiting system  becomes
\begin{equation}        
\left\{
\begin{array}{lll}
\rho_t+\nabla\cdot(\rho u)=0,\\
u_t+u\cdot\nabla u+\nabla (h(\rho)+\ln\rho)=0,
\end{array}
\right.\nonumber
\end{equation}
where $h(\rho)$ is determined through
$$h(\rho)=\int^\rho_{1}\frac{p'(\tau)}{\tau}d\tau,$$
for $\rho>0$.

\section{Proof of the main theorem--Theorem \ref{thm3.1}}

In this section, we are going to prove our main Theorem \ref{thm3.1}. We give some preliminaries in Subsection 3.1 and  the local well-posedness with the fixed small parameter and a priori assumption in Subsection 3.2. In the next three subsections, we will show the $L^2$ estimate, tangential estimates and normal estimates of the solutions, respectively. Finally, the proof of the main theorem is completed.

\subsection{Preliminary of Conormal Sobolev Spaces}

The proof of the main results in this paper relies on the conormal Sobolev spaces, so we give a short introduction here. To define the spaces, let us introduce the following tangential vector fields of the boundary
\begin{align}
Z_{1,2}=\partial_{x_1,x_2},\qquad Z_{3}=\psi(x_3)\partial_{x_3},\nonumber
\end{align}
where $\psi(x_3)=\frac{x_3}{1+x_3}$. Then, the conormal Sobolev space $H^m_{co}(\Omega)$ is defined as the set of functions $f(x)\in L^2(\Omega)$ such that the conormal derivatives of order at most $m$ of $f$ are also in $L^2(\Omega)$.

Next, for our purpose, we need to add another vector field $Z_0=\partial_t$ to the set of conormal derivatives. Setting
\begin{align}
Z^\alpha=Z_0^{\alpha_0}Z_1^{\alpha_1}Z_2^{\alpha_2}Z_3^{\alpha_3},\;\;\;\mbox{for}\;\;\alpha=(\alpha_0,\alpha_1,\alpha_2,\alpha_3), \nonumber
\end{align}
we define the conormal Sobolev spaces $H^m_{co}([0,T]\times\Omega)$ for an integer $m$ as
\begin{align}
H^m_{co}([0,T]\times\Omega)=\Big\{f:[0,T]\times \Omega\rightarrow \mathbb{R}^d|\;Z^\alpha f\in L^2([0,T]\times\Omega), \forall\;|\alpha|\leq m\Big\}.\nonumber
\end{align}
In our proof, we also need the following space
\begin{align}
X^m_{T}(\Omega)=\Big\{f:[0,T]\times \Omega\rightarrow \mathbb{R}^d|\;\partial^k_t f\in L^\infty([0,T];H^{m-k}_{co}(\Omega)), \forall \;k\leq m\Big\}.\nonumber
\end{align}
Introducing the semi-norms
\begin{align}
\|f\|_{\mathcal{H}^{m}_{co}}=\sum_{|\alpha|\leq m}\|Z^\alpha f(t)\|_{L^2(\Omega)},\nonumber
\end{align}
we can construct the following norm for $X^m_{T}(\Omega)$
\begin{align}
\|f\|^2_{\mathcal{H}^{m}_{co,T}}=\sup_{[0,T]}\|f\|^2_{\mathcal{H}^{m}_{co}}.\nonumber
\end{align}
Also, we define
\begin{align}
\|f\|_{\mathcal{H}^{m}}=\sum_{|\alpha|\leq m}\| \partial_t^{\alpha_0}\partial_1^{\alpha_1}\partial_2^{\alpha_2}\partial_3^{\alpha_3}f(t)\|_{L^2(\Omega)}.\nonumber
\end{align}
Now, we give some preliminary properties of the conormal derivatives. By a straightforward calculation, one can show that
\begin{align}
[Z_i,\partial_3]&=0, \qquad i=0,1,2,\nonumber\\
[Z_3,\partial_3]&=-\psi'(x_3)\partial_3, \label{tool1}\\
[Z^m_3,\partial_3]&=\sum_{\beta=0}^{m-1}\psi_{\beta,m}(x_3)Z_3^\beta\partial_3=\sum_{\beta=0}^{m-1}\psi^{\beta,m}(x_3)\partial_3Z_3^\beta,\label{tool2}\\
[Z^m_3,\partial_{33}]&=\sum_{\beta=0}^{m-1}\psi_{1,\beta,m}(x_3)Z_3^\beta\partial_3+\sum_{\beta=0}^{m-1}\psi_{2,\beta,m}(x_3)Z_3^\beta\partial_{33}\nonumber\\
&=\sum_{\beta=0}^{m-1}\psi^{1,\beta,m}(x_3)\partial_3Z_3^\beta+\sum_{\beta=0}^{m-1}\psi^{2,\beta,m}(x_3)\partial_{33}Z_3^\beta,\label{tool3}
\end{align}
where all the $\psi$'s  are bounded smooth functions. For the proof of the above equalities, we refer to \cite{Pad}.
Furthermore, the following trace estimate is standard
\begin{align}
|f|^2_{H^s(\partial\Omega)}\leq C(\|\nabla f\|_{\mathcal{H}^{m_1}_{co}}+\|f\|_{\mathcal{H}^{m_1}_{co}})\|f\|_{\mathcal{H}^{m_2}_{co}}, \qquad m_1+m_2\geq2s\geq0.
\end{align}

\subsection{Local well-posedness and a priori assumption} The proof of our main theorem, Theorem \ref{thm3.1}, is based on the local well-posedness of the system \eqref{2.4}, \eqref{2.7} and \eqref{2.8} for fixed $\epsilon$. In fact, for any small but fixed $\epsilon>0$, we can solve the Poisson equation in \eqref{2.4} and express $\phi$ in terms of $\rho$. Moreover, $\nabla \phi$ can be seen as a semi-linear term of $\rho$, thus the local existence of solution can be obtained by using similar method to the local well-posedness of the compressible Navier-Stokes equations with Navier boundary condition, c.f. \cite{Hoff}.

The proof of Theorem \ref{thm3.1} is based on this local well-posedness by using a bootstrap argument.  In the proof, we make the following  a priori assumption : \\
If $(\rho,u,\phi)$  is a smooth solution to \eqref{2.4}-\eqref{2.8} on $[0,T]$ , then
\begin{align}\label{apriori}\Lambda(t):=\|(\rho,u,\phi,\epsilon\nabla\phi)\|^2_{\mathcal{H}^3}+\epsilon^2\int_0^t\|\nabla u\|^2_{\mathcal{H}^3}d\tau\leq \epsilon^{2}, \end{align}
for $t\in [0, T].$

Under this a priori assumption, we will prove the following a priori estimate:

\begin{prop}\label{mainproposition} Let $(\rho,u,\phi)$  be a smooth solution to \eqref{2.4}-\eqref{2.8} on $[0,T]$ satisfying a priori assumption \eqref{apriori}, then the following estimate holds
\begin{align}\label{uniform}
\sup_{t\in[0, T]}\|(\rho,u,\phi,\epsilon\nabla\phi)(t)\|^2_{\mathcal{H}^3}+\epsilon^2\int_0^T\|\nabla u\|^2_{\mathcal{H}^3}\leq C\epsilon^{2K-10},
\end{align}
for a constant $C$ independent of $\epsilon$.
\end{prop}

Since $K>6$ ($2K-10>2$), the estimate \eqref{uniform} closes the bootstrap argument, and completes the proof of Theorem \ref{thm3.1}.
The remaining part of this paper is devoted to prove Proposition \ref{mainproposition}.

\subsection{$L^2$ estimates}
In this subsection, we shall first give the $L^2$ estimates of the solution.
\begin{prop} \label{L2}
Under the assumptions of Theorem \ref{thm3.1}, if $(\rho,u,\phi)$ is a smooth solution of \eqref{2.4}-\eqref{2.8} on $[0,T]$ satisfying a priori assumption \eqref{apriori},
then the following estimate holds
\begin{align}
\|(\rho&,u,\phi,\epsilon\nabla\phi)(t)\|^2+\epsilon^2\int_0^t\|\nabla u\|^2d\tau\leq C\epsilon^{2K}\nonumber
\end{align}
for $t\in [0, T]$.
\end{prop}
\begin{proof} Rewrite  equation \eqref{2.4} as
\begin{equation}   \label{3.1}       
\left\{
\begin{array}{lll}
\partial_t\rho+(u_a+u)\cdot\nabla\rho+(\rho_a+\rho)\nabla\cdot u+u\cdot\nabla\rho_a+\rho\nabla\cdot u_a=\epsilon^KR_\rho,\\
(\rho_a+\rho)\big(\partial_t u+(u_a+u)\cdot\nabla u+u\cdot\nabla u_a\big)+(\nabla \rho-\frac{\nabla\rho_a}{\rho_a}\rho)\\
\qquad\qquad=(\rho_a+\rho)\nabla\phi+\mu\epsilon^2\triangle u+(\mu+\nu)\epsilon^2\nabla\nabla\cdot u+H(\rho, \rho_a, u_a)+\epsilon^KR_u,\\
\epsilon^2\triangle\phi=\rho-e^{-\phi_a}(e^{-\phi}-1)+\epsilon^{K+1}R_\phi,
\end{array}
\right.
\end{equation}
where we have set $T^i=1$ without loss of generality and
\begin{align}
H(\rho, \rho_a, u_a)=-\frac{\mu\epsilon^2\rho}{\rho_a}\triangle u_a-\frac{(\mu+\nu)\epsilon^2\rho}{\rho_a}\nabla\nabla\cdot u_a.\nonumber
\end{align}
Then $\eqref{3.1}_1\times \frac{1}{\rho_a+\rho}\rho+\eqref{3.1}_2\cdot u$ gives
\begin{align} \label{3.2}
\frac{1}{\rho_a+\rho}\frac{1}{2}&\partial_t\rho^2+\frac{1}{\rho_a+\rho}\frac{1}{2}(u_a+u)\cdot\nabla \rho^2+\rho\nabla\cdot  u\nonumber\\
&
+\frac{1}{\rho_a+\rho}\rho u\cdot\nabla\rho_a+\frac{1}{\rho_a+\rho}\rho^2\nabla\cdot u_a\nonumber\\
&+(\rho_a+\rho)\big(\frac{1}{2}\partial_t  |u|^2+\frac{1}{2}(u_a+u)\cdot\nabla | u|^2+ u\cdot\nabla u_a\cdot u\big)\nonumber\\
&+(u\cdot\nabla \rho -\frac{\rho}{\rho_a}u\cdot\nabla\rho_a \big)\nonumber\\
&=(\rho_a+\rho)u\cdot\nabla\phi+\mu\epsilon^2u\cdot\triangle u+(\mu+\nu)\epsilon^2u\cdot\nabla\nabla\cdot u\nonumber\\
&+ H(\rho, \rho_a, u_a)\cdot u+\epsilon^K\frac{1}{\rho_a+\rho}\rho R_\rho+\epsilon^Ku\cdot R_u.
\end{align}
Integrating \eqref{3.2} over $[0,T]\times \Omega$, and integrating by parts, we get
\begin{align} \label{3.3}
&\frac{1}{2}\int_\Omega\frac{1}{\rho_a+\rho}\rho^2-\frac{1}{2}\int_0^T\int_\Omega\rho^2\partial_t\frac{1}{\rho_a+\rho}
-\frac{1}{2}\int_0^T\int_\Omega\nabla\cdot\Big(\frac{u_a+u}{\rho_a+\rho}\Big)\rho^2\nonumber\\
&+\int_0^T\int_\Omega\frac{1}{\rho_a+\rho}\rho^2\nabla\cdot u_a+\int_0^T\int_\Omega(\rho_a+\rho)u\cdot\nabla u_a\cdot u+\frac{1}{2}\int_\Omega(\rho_a+\rho)|u|^2\nonumber\\
&-\frac{1}{2}\int_0^T\int_\Omega|u|^2\partial_t(\rho_a+\rho)
-\frac{1}{2}\int_0^T\int_\Omega|u|^2\nabla\cdot\Big((\rho_a+\rho)(u_a+u)\Big)\nonumber\\
&+\int_0^T\int_\Omega\frac{1}{\rho_a+\rho}\rho u\cdot\nabla\rho_a-\int_0^T\int_\Omega\frac{\rho}{\rho_a}u\cdot\nabla\rho_a \nonumber\\
=&\;\frac{1}{2}\int_\Omega\Big(\frac{1}{\rho_a+\rho}\rho^2\Big)(0)+\frac{1}{2}\int_\Omega\Big((\rho_a+\rho)|u|^2\Big)(0)\nonumber\\
&+\int_0^T\int_\Omega(\rho_a+\rho)u\cdot\nabla\phi+\mu\epsilon^2\int_0^T\int_\Omega u\cdot \triangle u+(\mu+\nu)\epsilon^2\int_0^T\int_\Omega u\cdot \nabla\nabla\cdot u\nonumber\\
&+ \int_0^T\int_\Omega H(\rho, \rho_a, u_a)u
+\epsilon^K\int_0^T\int_\Omega\frac{1}{\rho_a+\rho}\rho R_\rho+\epsilon^K\int_0^T\int_\Omega u\cdot R_u.
\end{align}
Since
\begin{align}
-\mu\epsilon^2\int_0^T\int_\Omega u\triangle u&=\mu\epsilon^2\int_0^T\int_\Omega |\nabla u|^2+\mu\epsilon^2\int_0^T\int_{\partial\Omega} |u_y|^2,\nonumber
\end{align}
and
\begin{align}
-(\mu+\nu)\epsilon^2\int_0^T\int_\Omega u\cdot \nabla\nabla\cdot u&=(\mu+\nu)\epsilon^2\int_0^T\int_\Omega |\nabla\cdot u|^2,\nonumber
\end{align}
we get from \eqref{3.3} that
\begin{align} \label{3.4}
&\|\rho\|^2+\|u\|^2+\epsilon^2\int_0^T\|\nabla u\|^2+\epsilon^2\int_0^T|u_h|_{L^2(\partial\Omega)}^2\nonumber\\
\leq & \int_0^T\int_\Omega\rho^2(\partial_t+(u_a+u)\cdot\nabla)\frac{1}{\rho_a+\rho}+\int_0^T\int_\Omega\rho^2\frac{1}{(\rho_a+\rho)}\nabla\cdot(u_a+u)\nonumber\\
&+\int_0^T\int_\Omega|u|^2(\partial_t+(u_a+u)\cdot\nabla)(\rho_a+\rho)+\int_0^T\int_\Omega|u|^2(\rho_a+\rho)\nabla\cdot(u_a+u)\nonumber\\
&+\int_0^T\int_\Omega(\frac{1}{\rho_a+\rho}-\frac{1}{\rho_a})\rho u\cdot\nabla\rho_a+\int_0^T\int_\Omega(\rho_a+\rho)u\cdot\nabla\phi \nonumber\\
&+C\int_0^T(\|\rho\|^2+\|u\|^2)+C\epsilon^{2K}.
\end{align}
We need to consider the estimates of the right-hand side terms of \eqref{3.4}. First, we have
\begin{align}
\|(\partial_t+(u_a+u)\cdot\nabla)(\rho_a+\rho)\|_{L^\infty}\leq & \|(\partial_t+u_a\cdot\nabla)\rho_a\|_{L^\infty}+\|u\cdot\nabla\rho_a\|_{L^\infty}\nonumber\\
&+\|(\partial_t+(u_a+u)\cdot\nabla)\rho\|_{L^\infty}\nonumber.
\end{align}
By the properties of the approximation solutions and the a priori assumptions, we find that
\begin{align}
\|(\partial_t+u_a\cdot\nabla)\rho_a\|_{L^\infty}&\leq C,\nonumber\\
\|(\partial_t+(u_a+u)\cdot\nabla)\rho\|_{L^\infty}&\leq C(\|\partial_t\rho\|_{L^\infty}+\|\nabla\rho\|_{L^\infty})\leq C.\nonumber
\end{align}
As $u_3$ vanishes on the boundary, we have $|u_3|\leq x_3\|\nabla u\|_{L^\infty}$. Hence,
\begin{align}
\|u\cdot\nabla\rho_a\|_{L^\infty}&\leq \|u_y\cdot\nabla_y\rho_{a}\|_{L^\infty}+\|u_3\partial_3\rho_{a}\|_{L^\infty}\nonumber\\
&\leq C+\|\nabla u\|_{L^\infty}\|x_3\partial_3\rho_{a}\|\nonumber\\
&\leq C+\|\nabla u\|_{L^\infty}\|z\partial_z\rho_{a}\|\leq C,\nonumber
\end{align}
where we have used the fast decay property of the boundary layer profiles in $\rho_a$.
So we get
\begin{align} \label{L^infty-1}
\|(\partial_t+(u_a+u)\cdot\nabla)(\rho_a+\rho)\|_{L^\infty}&\leq C.
\end{align}
Next, as the $U^0=0$, it is easy to observe that
\begin{align} \label{L^infty-2}
\|\frac{1}{(\rho_a+\rho)}\nabla\cdot(u_a+u)\|_{L^\infty}+\|(\rho_a+\rho)\nabla\cdot(u_a+u)\|_{L^\infty}\leq C.
\end{align}
Moreover, by the a priori assumptions, one has
\begin{align}
\Big|\int_0^T\int_\Omega(\frac{1}{\rho_a+\rho}-\frac{1}{\rho_a})\rho u\cdot\nabla\rho_a \Big|&\leq \int_0^T \|\rho\|_{L^\infty}\|\rho\| \|u\| \|\nabla\rho_a\|_{L^\infty}\nonumber\\
&\leq \int_0^T \frac{1}{\epsilon}\|\rho\|_{L^\infty}(\|\rho\|^2+ \|u\|^2)\nonumber\\
&\leq C\int_0^T (\|\rho\|^2+ \|u\|^2).\nonumber
\end{align}
Using the above estimates, we can reduce \eqref{3.4} into
\begin{align} \label{3.5}
\|\rho\|^2&+\|u\|^2+\epsilon^2\int_0^T\|\nabla u\|^2+\epsilon^2\int_0^T|u_h|_{L^2(\partial\Omega)}^2\nonumber\\
&\leq \int_0^T\int_\Omega(\rho_a+\rho)u\cdot\nabla\phi+C\int_0^T(\|\rho\|^2+\|u\|^2)+C\epsilon^{2K}.
\end{align}
Let us consider the first term on the right-hand side of \eqref{3.5}. Using $\eqref{3.1}_1$ and integrating by parts, we get
\begin{align} \label{3.6}
\int_0^T\int_\Omega(\rho_a+\rho)u\cdot\nabla\phi&=-\int_0^T\int_\Omega\nabla\cdot\big((\rho_a+\rho)u\big)\phi\nonumber\\
&=\int_0^T\int_\Omega\big(\partial_t\rho+\nabla\cdot(\rho u_a)-\epsilon^KR_\rho\big)\phi\nonumber\\
&=\int_0^T\int_\Omega\phi\partial_t\rho-\rho u_a\cdot\nabla\phi-\epsilon^KR_\rho\phi\nonumber\\
&:=\sum_{j=1}^3I_j.
\end{align}
By $\eqref{3.1}_3$, we can estimate $I_1$ as
\begin{align}
I_1=&\int_0^T\int_\Omega \phi\Big(\epsilon^2\triangle\partial_t\phi+\partial_t\big(e^{-\phi_a}(e^{-\phi}-1)\big)-\epsilon^{K+1}\partial_tR_\phi\Big)\nonumber\\
=&-\frac{1}{2}\epsilon^2\int_\Omega|\nabla\phi|^2+\frac{1}{2}\epsilon^2\int_\Omega|\nabla\phi(0)|^2
-\frac{1}{2}\int_\Omega e^{-\phi_a}(1+h(\phi))\phi^2\nonumber\\
&+\frac{1}{2}\int_\Omega e^{-\phi_a}(1+h(\phi))\phi^2(0)-\frac{1}{2}\int_0^T\int_\Omega \partial_t\big(e^{-\phi_a}(1+h(\phi))\big)\phi^2\nonumber\\
&-\epsilon^{K+1}\int_0^T\int_\Omega \phi\partial_tR_\phi,\nonumber
\end{align}
where
$$h=h_0(\phi):=-\frac{e^{-\phi}-1+\phi}{\phi}.$$
Since
$$\|\partial_t\big(e^{-\phi_a}(1+h(\phi))\big)\|_{L^\infty}\leq C(\|\phi\|_{L^\infty}+\|Z\phi\|_{L^\infty})\leq C,$$
we have that,
\begin{align}  \label{3.7}
I_1
&\leq -\frac{1}{2}\epsilon^2\int_\Omega|\nabla\phi|^2-\frac{1}{2}\int_\Omega e^{-\phi_a}(1+h(\phi))\phi^2+\int_0^T\|\phi\|^2+C\epsilon^{2K}.
\end{align}
For $I_2$, by integrating by parts, we have
\begin{align} \label{3.8}
I_2=&-\int_0^T\int_\Omega\rho u_a\cdot\nabla\phi\nonumber\\
=&\int_0^T\int_\Omega(\epsilon^2\triangle\phi+e^{-\phi_a}(e^{-\phi}-1)-\epsilon^{K+1}R_\phi) u_a\cdot\nabla\phi\nonumber\\
\leq &\;C\int_0^T\epsilon^2\|\nabla\phi\|^2+\|\phi\|^2+C\epsilon^{2K}.
\end{align}
The estimate of $I_3$ is trivial since
\begin{align} \label{3.9}
I_3&\leq C\int_0^T\|\phi\|^2+C\epsilon^{2K}.
\end{align}
Combining the estimates in \eqref{3.6}-\eqref{3.9}, we get
\begin{align}
\int_0^T\int_\Omega(\rho_a+\rho)u\cdot\nabla\phi\leq &-\frac{1}{2}\epsilon^2\int_\Omega|\nabla\phi|^2-\frac{1}{2}\int_\Omega e^{-\phi_a}(1+h(\phi))\phi^2\nonumber\\
&+\int_0^T(\epsilon^2\|\nabla\phi\|^2+\|\phi\|^2)+C\epsilon^{2K}.\nonumber
\end{align}
This together with \eqref{3.5} yields
\begin{align}
\big(\|\rho\|^2&+\|u\|^2+\|\phi\|^2+\epsilon^2\|\nabla\phi\|^2\big)+\epsilon^2\int_0^T\|\nabla u\|^2+\epsilon^2\int_0^T|u_h|_{L^2(\partial\Omega)}^2\nonumber\\
&\leq C\int_0^T(\|\rho\|^2+\|u\|^2+\|\phi\|^2+\epsilon^2\|\nabla\phi\|^2)+C\epsilon^{2K}.\nonumber
\end{align}
Applying Gronwall's inequality, we can get the desired $L^2$ estimate. Thus the proof of Proposition \ref{L2} is completed.
\end{proof}
\subsection{Tangential estimates}
In this subsection, we turn to the estimates of the high-order tangential derivatives of the solution.
\begin{prop} \label{tan}
For $0\leq j\leq 3$, the following estimates hold
\begin{align}
\|Z^j(\rho&,u,\phi,\epsilon\nabla\phi)(t)\|^2+\epsilon^2\int_0^t\|\nabla Z^ju(\tau)\|^2d\tau\leq C\epsilon^{2K-2j},\nonumber
\end{align}
under the assumptions of Proposition \ref{L2}.
\end{prop}
\begin{proof}
From Proposition \ref{L2}, the results in Proposition \ref{tan} hold for $j=0$. Now, assuming that they  hold for $j\leq 2$, we deal with the case $j=3$. For $|\alpha|=3$, apply $Z^\alpha$ to the equation \eqref{3.1}, and isolate the highest-order terms as follows:
\begin{equation}   \label{tan1}       
\left\{
\begin{array}{lll}
\partial_tZ^\alpha\rho+(u_a+u)\cdot\nabla Z^\alpha\rho+(\rho_a+\rho)\nabla\cdot Z^\alpha u
+Z^\alpha u\cdot\nabla\rho_a+Z^\alpha\rho\nabla\cdot u_a\\
\qquad\qquad\qquad\qquad=\epsilon^KZ^\alpha R_\rho+\mathcal{C}_\rho,\\
(\rho_a+\rho)\big(\partial_t Z^\alpha u+(u_a+u)\cdot\nabla Z^\alpha u+Z^\alpha u\cdot\nabla u_a\big)
+\big(\nabla Z^\alpha\rho-\frac{\nabla\rho_a}{\rho_a}Z^\alpha\rho\big)\\
\qquad\qquad\qquad\qquad=(\rho_a+\rho)\nabla Z^\alpha\phi+\mu\epsilon^2\triangle Z^\alpha u+(\mu+\nu)\epsilon^2\nabla\nabla\cdot Z^\alpha u\\
\qquad\qquad\qquad\qquad\qquad\qquad+Z^\alpha H(\rho, \rho_a, u_a)+\epsilon^KZ^\alpha R_u+\mathcal{C}_u,\\
\epsilon^2\triangle Z^\alpha\phi=Z^\alpha\rho+e^{-\phi_a}Z^\alpha\phi(1+h)+\epsilon^{K+1}Z^\alpha R_\phi+\mathcal{C}_\phi,
\end{array}
\right.
\end{equation}
where $h=h_1:=e^{-\phi}-1$,
and the commutators $\mathcal{C}_\rho,\mathcal{C}_u,\mathcal{C}_\phi$  given by
\begin{align}
\mathcal{C}_\rho=&-[Z^\alpha,(u_a+u)\cdot\nabla]\rho-[Z^\alpha,(\rho_a+\rho)\nabla\cdot]u-[Z^\alpha,\nabla\rho_a\cdot] u-[Z^\alpha,\nabla\cdot u_a]\rho,\nonumber\\
\mathcal{C}_u=&-[Z^\alpha,\rho_a+\rho]\partial_tu-[Z^\alpha,(\rho_a+\rho)(u_a+u)\cdot\nabla]u-[Z^\alpha,(\rho_a+\rho)\nabla u_a\cdot]u\nonumber\\
&-[Z^\alpha,\nabla]\rho+[Z^\alpha,\frac{\nabla\rho_a}{\rho_a}]\rho+[Z^\alpha,(\rho_a+\rho)\nabla]\phi+\mu\epsilon^2[Z^\alpha,\triangle]u\nonumber\\
&+(\mu+\nu)\epsilon^2[Z^\alpha,\nabla\nabla\cdot]u,\nonumber\\
\mathcal{C}_\phi=&-[Z^\alpha,\epsilon^2\triangle]\phi-Z^\alpha\big(e^{-\phi_a}(e^{-\phi}-1)\big)-e^{-\phi_a}Z^\alpha\phi e^{-\phi}.\nonumber
\end{align}
Then, a similar energy estimate as Proposition \ref{L2} gives
\begin{align}  \label{tan2}
\|Z^\alpha\rho\|^2+\|Z^\alpha u\|^2\leq& \int_0^T\int_\Omega(\rho_a+\rho)Z^\alpha u\cdot\nabla Z^\alpha\phi
+\int_0^T\int_\Omega\mu\epsilon^{2}Z^\alpha u\triangle Z^\alpha u\nonumber\\
&+\int_0^T\int_\Omega(\mu+\nu)\epsilon^{2}Z^\alpha u\nabla\nabla\cdot Z^\alpha u
+\int_0^T\int_\Omega Z^\alpha uZ^\alpha H(\rho, \rho_a, u_a)\nonumber\\
&+\int_0^T\int_\Omega\frac{1}{\rho_a+\rho}Z^\alpha\rho\mathcal{C}_\rho
+\int_0^T\int_\Omega Z^\alpha u\mathcal{C}_u\nonumber\\
&+C\int_0^T(\|Z^\alpha\rho\|^2+\|Z^\alpha u\|^2)+C\epsilon^{2K-6}\nonumber\\
:=&\sum_{k=1}^6J_k+C\int_0^T(\|Z^\alpha\rho\|^2+\|Z^\alpha u\|^2)+C\epsilon^{2K-6}.
\end{align}
First let's consider $J_1$. Integrating by parts, we get
\begin{align}
\int_{\mathbb{R}^3_+}(\rho_a+\rho)Z^\alpha u\cdot\nabla Z^\alpha\phi=-\int_{\mathbb{R}^3_+}Z^\alpha\phi\nabla\cdot\big((\rho_a+\rho)Z^\alpha u\big).\nonumber
\end{align}
Using the first equation of \eqref{tan1}, we have
\begin{align} \label{tan3}
J_1=&\int_0^T\int_{\mathbb{R}^3_+}Z^\alpha\phi\partial_tZ^\alpha\rho+\int_0^T\int_{\mathbb{R}^3_+}Z^\alpha\phi\nabla\cdot[(u_a+u)Z^\alpha\rho])\nonumber\\
&-\int_0^T\int_{\mathbb{R}^3_+}Z^\alpha\phi(Z^\alpha\rho\nabla\cdot u +Z^\alpha u\cdot\nabla \rho+\epsilon^KZ^\alpha R_\rho+\mathcal{C}_\rho)\nonumber\\
:=&\int_0^T\sum_{k=1}^3J_{1k}.
\end{align}
By the last equation of \eqref{tan1}, we have
\begin{align}
J_{11}=&\int_{\mathbb{R}^3_+}Z^\alpha\phi(\epsilon^2\partial_t\triangle Z^\alpha\phi
-\partial_t(e^{-\phi_a}Z^\alpha\phi(1+h))+\epsilon^{K+1}\partial_tZ^\alpha R_\phi+\partial_t\mathcal{C}_\phi)\nonumber\\
=&-\epsilon^2\partial_t\int_{\mathbb{R}^3_+}|\nabla Z^\alpha\phi|^2
-\frac{1}{2}\partial_t\int_{\mathbb{R}^3_+}(e^{-\phi_a}|Z^\alpha\phi|^2(1+h))\nonumber\\
&-\frac{1}{2}\int_{\mathbb{R}^3_+}\partial_t(e^{-\phi_a}(1+h))|Z^\alpha\phi|^2
-\int_{\mathbb{R}^3_+}Z^\alpha\phi(\epsilon^{K+1}\partial_tZ^\alpha R_\phi+\partial_t\mathcal{C}_\phi).\nonumber
\end{align}
Clearly, the last two terms can be bounded by
\begin{align}
C(1+\|\partial_t\phi\|_{L^\infty})\|Z^\alpha\phi\|^2+(\epsilon^{2K+2}+\|\partial_t\mathcal{C}_\phi\|^2).\nonumber
\end{align}
Thus
\begin{align}
\int_0^TJ_{11}\leq&-\epsilon^2\int_{\mathbb{R}^3_+}|\nabla Z^\alpha\phi|^2
-\int_{\mathbb{R}^3_+}(e^{-\phi_a}|Z^\alpha\phi|^2(1+h))\nonumber\\
+&C\int_0^T\|Z^\alpha\phi\|^2+\int_0^T\|\partial_t\mathcal{C}_\phi\|^2+C\epsilon^{2K-6}.\nonumber
\end{align}
The second term on the right-hand side of \eqref{tan3} can be estimated as
\begin{align}
J_{12}=&\int_{\mathbb{R}^3_+}Z^\alpha\phi\nabla\cdot[(u_a+u)Z^\alpha\rho]\nonumber\\
=&-\int_{\mathbb{R}^3_+}(u_a+u)\cdot\nabla Z^\alpha\phi(\epsilon^2\triangle Z^\alpha\phi
-e^{-\phi_a}Z^\alpha\phi(1+h)+\epsilon^{K+1}Z^\alpha R_\phi+\mathcal{C}_\phi)\nonumber\\
\leq& \;C\epsilon^2(1+\|\nabla u\|_{L^\infty})\|\nabla Z^\alpha\phi\|^2
+C(1+\|\nabla \phi\|_{L^\infty})\|Z^\alpha\phi\|^2\nonumber\\
&+\frac{1}{\epsilon^2}\|\mathcal{C}_\phi\|^2+C\|\epsilon\nabla Z^\alpha\phi\|^2+C\epsilon^{2K}.\nonumber
\end{align}
For $J_{13}$, we easily have
\begin{align}
J_{13}\leq& \;\|Z^\alpha\phi\| \|Z^\alpha\rho\|\|\nabla\cdot u\|_{L^\infty}+\|Z^\alpha\phi\| \|Z^\alpha u\|\|\nabla \rho\|_{L^\infty}\nonumber\\
&+C \|Z^\alpha\phi\|^2+\|\mathcal{C}_\rho\|^2+\epsilon^{2K}\nonumber\\
\leq& \;C(\|Z^\alpha\rho\|^2+\|Z^\alpha\phi\|^2+\|Z^\alpha u\|^2)+\|\mathcal{C}_\rho\|^2+C\epsilon^{2K}.\nonumber
\end{align}
Collecting all the above discussion about $J_{1k}$, we obtain
\begin{align} \label{tan4}
J_1&+\epsilon^2\int_{\mathbb{R}^3_+}|\nabla Z^\alpha\phi|^2
+\int_{\mathbb{R}^3_+}(e^{-\phi_a}(1+h)|Z^\alpha\phi|^2)\nonumber\\
\leq &\;C\int_0^T(\|Z^\alpha\phi\|^2+\epsilon^2\|\nabla Z^\alpha\phi\|^2+\|Z^\alpha u\|^2+\|Z^\alpha\rho\|^2)\nonumber\\
&+\int_0^T(\|\partial_t\mathcal{C}_\phi\|^2
+\frac{1}{\epsilon^2}\|\mathcal{C}_\phi\|^2+\|\mathcal{C}_\rho\|^2)+C\epsilon^{2K-6}.
\end{align}
Next we turn to estimate the second term on the right-hand side of \eqref{tan2}.
Integrating by parts, we have
\begin{align}
J_2=-\int_0^T\int_{\mathbb{R}^3_+}\mu\epsilon^{2}\nabla Z^\alpha u\nabla Z^\alpha u-\int_0^T\int_{\mathbb{R}^2}\mu\epsilon^{2}\partial_3 Z^\alpha u_yZ^\alpha u_y.\nonumber
\end{align}
Using the boundary condition \eqref{2.8}, we have
\begin{align}
\partial_3 Z^\alpha u_y= Z^\alpha \partial_3u_y+[\partial_3,Z^\alpha]u_y= Z^\alpha u_y+[\partial_3,Z^\alpha]u_y.\nonumber
\end{align}
By the commutator's property and the trace theorem in Section 3, we have
\begin{align}
-\int_{\mathbb{R}^2}\mu\epsilon^2[\partial_3,Z^\alpha]u_yZ^\alpha u_y\leq&\; \mu\epsilon^2|[\partial_3,Z^\alpha]u_y|_{L^2(\mathbb{R}^2)}|Z^\alpha u_y|_{L^2(\mathbb{R}^2)}\nonumber\\
\leq& \;\mu\epsilon^2|[\partial_3,Z^\alpha]u_y|_{L^2(\mathbb{R}^2)}(\|\nabla Z^\alpha u_y\|+\|Z^\alpha u_y\|)\nonumber\\
\leq&\; \frac{\mu}{2}\epsilon^2\|\nabla Z^\alpha u_y\|^2+C\|Z^\alpha u_y\|^2+C\epsilon^2|[\partial_3,Z^\alpha]u_y|^2_{L^2(\mathbb{R}^2)}\nonumber\\
\leq& \;\frac{\mu}{2}\epsilon^2\|\nabla Z^\alpha u_y\|^2+C\|Z^\alpha u_y\|^2+C\epsilon^{2}|Z^2u_y|^2\nonumber\\
\leq& \;\frac{\mu}{2}\epsilon^2\|\nabla Z^\alpha u_y\|^2+C\epsilon^2\|\nabla Z^2u_y\|^2+C\|u\|^2_{\mathcal{H}^2_{co}}.\nonumber
\end{align}
Therefore,
\begin{align} \label{tan5}
J_2&+\frac{\mu}{2}\epsilon^{2}\int_0^T\|\nabla Z^\alpha u\|^2+\mu\epsilon^{2}\int_0^T |Z^\alpha u_y|_{L^2(\mathbb{R}^2)}\leq C\epsilon^{2K-4}.
\end{align}
Similarly, we have for $J_3$,
\begin{align} \label{tan6}
\int_0^T\int_{\mathbb{R}^3_+}(\mu+\nu)\epsilon^2\nabla\nabla\cdot Z^\alpha uZ^\alpha u=-\int_0^T\int_{\mathbb{R}^3_+}(\mu+\nu)\epsilon^2\nabla\cdot Z^\alpha u\nabla\cdot Z^\alpha u.
\end{align}
For $J_4$, by using the assumptions for $j=0,1,2$, we get
\begin{align} \label{tan7}
J_4&=\int_0^T\int_\Omega Z^\alpha uZ^\alpha H(\rho, \rho_a, u_a)\nonumber\\
&\leq C\int_0^T\int_\Omega\epsilon^{2}Z^\alpha u\cdot Z^\alpha \big(\frac{\triangle u_a}{\rho_a}\rho\big)+\epsilon^{2}Z^\alpha u\cdot Z^\alpha \big(\frac{\nabla\nabla\cdot u_a}{\rho_a}\rho\big)\nonumber\\
&\leq C\int_0^T(\|Z^\alpha u\|^2+\|Z^\alpha \rho\|^2)+C\epsilon^{2K-6}.
\end{align}
Thus, collecting \eqref{tan4}-\eqref{tan7}, one has
\begin{align} \label{tan8}
\big(\|Z^\alpha\rho\|^2&+\|Z^\alpha u\|^2+\|Z^\alpha \phi\|^2+\epsilon^2\|\nabla Z^\alpha\phi\|^2\big)+\frac{\mu}{2}\epsilon^2\int_0^T\|\nabla Z^\alpha u\|^2
\nonumber\\
\leq&\; J_5+J_6+C\int_0^T(\|\partial_t\mathcal{C}_\phi\|^2
+\frac{1}{\varepsilon^2}\|\mathcal{C}_\phi\|^2+\|\mathcal{C}_\rho\|^2)\nonumber\\
&+C\int_0^T(\|Z^\alpha\rho\|^2+\|Z^\alpha u\|^2+\|Z^\alpha \phi\|^2+\epsilon^2\|\nabla Z^\alpha\phi\|^2)+C\epsilon^{2K-6}.
\end{align}
Now, it remains to deal with the estimates of the commutators. Recalling the expression of the commutator term $J_5$, we have
\begin{align} \label{tan9}
J_5&=\int_0^T\int_\Omega\frac{1}{\rho_a+\rho}Z^\alpha\rho\mathcal{C}_\rho\leq \int_0^T\|Z^\alpha\rho\|^2+\int_0^T\|\mathcal{C}_\rho\|^2\nonumber\\
&\leq \int_0^T\|Z^\alpha\rho\|^2+\int_0^T\|[Z^\alpha,(u_a+u)\cdot\nabla]\rho\|^2\nonumber\\
&\;\;\;+\int_0^T\|[Z^\alpha,(\rho_a+\rho)\nabla\cdot]u\|^2+\int_0^T\|[Z^\alpha,\nabla\rho_a\cdot] u\|^2+\int_0^T\|[Z^\alpha,\nabla\cdot u_a]\rho\|^2\nonumber\\
&:=\int_0^T\|Z^\alpha\rho\|^2+\sum_{k=1}^4J_{5k}.
\end{align}
First in view of \eqref{tool2}, we have
\begin{align}
J_{51}=&\int_0^T\|[Z^\alpha,(u_a+u)\cdot]\nabla\rho\|^2+\int_0^T\|(u_a+u)[Z^\alpha,\nabla]\rho\|^2\nonumber\\
\leq& \int_0^T\|Z^3(u_a+u)\nabla\rho\|^2+\|Z^2(u_a+u)Z\nabla\rho\|^2+\|Z(u_a+u)Z^2\nabla\rho\|^2\nonumber\\
&+ \int_0^T\|(u_a+u)\sum_{\beta=0}^2\psi^{\beta,3}(x_3)\partial_3Z^\beta\rho\|^2\nonumber\\
\leq& \int_0^T\|\nabla Z^3u_a\|^2_{L^\infty}\|Z\rho\|^2+\|Z^3u\|^2\|\nabla \rho\|_{L^\infty}^2
+\|\nabla Z^2u_a\|^2_{L^\infty}\|Z^2\rho\|^2\nonumber\\
&+\int_0^T\|Z^2u\|^2_{L^6}\|Z\nabla\rho\|_{L^6}^2
+\|\nabla Zu_a\|^2_{L^\infty}\|Z^3\rho\|^2+\|\nabla Zu\|^2_{L^\infty}\|Z^3 \rho\|^2\nonumber\\
&+ \int_0^T\|\nabla u_a\|^2_{L^\infty}\|\rho\|_{\mathcal{H}^3_{co}}^2+\|\nabla u\|_{L^\infty}^2\|\rho\|_{\mathcal{H}^3_{co}}^2\nonumber\\
\leq& \;C\int_0^T(\|Z^3\rho\|^2+\|Z^3u\|^2)+\int_0^T\|Zu\|_{H^3}^2\|Z^3\rho\|^2+C\epsilon^{2K-6}.\nonumber
\end{align}
$J_{52}$ can be estimated by standard commutator estimates as follows
\begin{align}
J_{52}&=\int_0^T\|[Z^\alpha,(\rho_a+\rho)\nabla\cdot]u\|^2\nonumber\\
&=\int_0^T\|[Z^\alpha,(\rho_a+\rho)]\nabla\cdot u\|^2+\int_0^T\|(\rho_a+\rho)[Z^\alpha,\nabla\cdot]u\|^2\nonumber\\
&\leq \int_0^T\|Z^3(\rho_a+\rho)\nabla\cdot u\|^2+\|Z^2(\rho_a+\rho)Z\nabla\cdot u\|^2+\|Z(\rho_a+\rho)Z^2\nabla\cdot u\|^2\nonumber\\
&\;\;+\int_0^T\|(\rho_a+\rho)\sum_{\beta=0}^2\psi^{\beta,3}(x_3)\partial_3Z^\beta u\|^2\nonumber\\
&\leq \int_0^T\|Z^3\rho_a\|_{L^\infty}^2\|\nabla\cdot u\|^2+\|Z^3\rho\|^2\|\nabla\cdot u\|_{L^\infty}^2+\|Z^2\rho_a\|_{L^\infty}^2\|Z\nabla\cdot u\|^2\nonumber\\
&\;\;+\int_0^T\|Z^2\rho\|^2\|Z\nabla\cdot u\|_{L^\infty}^2+\|Z\rho_a\|_{L^\infty}^2\|Z^2\nabla\cdot u\|^2+\|Z\rho\|^2_{L^\infty}\|Z^2\nabla\cdot u\|^2\nonumber\\
&\;\;+\int_0^T\|(\rho_a+\rho)\|^2_{L^\infty}\sum_{\beta=0}^2\|\nabla Z^\beta u\|^2\nonumber\\
&\leq C\int_0^T\|Z^3\rho\|^2+C\epsilon^{2K-6}.\nonumber
\end{align}
For $J_{53}$, noticing that $|\nabla\rho_a|=O(\frac{1}{\epsilon})$, we have
\begin{align}
J_{53}&=\int_0^T\|[Z^\alpha,\nabla\rho_a\cdot] u\|^2\leq\int_0^T\sum_{|\beta|+|\gamma|=3,|\beta|>0}\|Z^\beta \nabla\rho_a Z^\gamma u\|\nonumber\\
&\leq \int_0^T \sum_{|\beta|+|\gamma|=3,|\beta|>0}\|Z^\beta \nabla\rho_a\|^2_{L^\infty} \|Z^\gamma u\|^2\leq \frac{C}{\epsilon^2}\|u\|^2_{\mathcal{H}^2_{co}}\leq C\epsilon^{2K-6}.\nonumber
\end{align}
Finally, one has
\begin{align}
J_{54}&=\int_0^T\|[Z^\alpha,\nabla\cdot u_a]\rho\|^2\leq C\int_0^T\|\rho\|_{\mathcal{H}^2_{co}}^2\leq C\epsilon^{2K-4}.\nonumber
\end{align}
Collecting the above estimates, we get
\begin{align} \label{tan10}
J_{5}
&\leq C\int_0^T(\|Z^3\rho\|^2+\|Z^3u\|^2)+ \int_0^T\|Zu\|_{H^3}^2\|Z^3\rho\|^2+C\epsilon^{2K-6}.
\end{align}
For $J_6$, we get from the expression of $\mathcal{C}_u$ that
\begin{align} \label{tan11}
J_6=&\int_0^T\int_\Omega Z^\alpha u\mathcal{C}_u
\leq \int_0^T\|Z^\alpha u\|^2+\int_0^T\|[Z^\alpha,\rho_a+\rho]\partial_tu\|^2\nonumber\\
&+ \int_0^T\|[Z^\alpha,(\rho_a+\rho)(u_a+u)\cdot\nabla]u\|^2+\int_0^T\|[Z^\alpha,(\rho_a+\rho)\nabla u_a\cdot]u\|^2\nonumber\\
&+\int_0^T\|[Z^\alpha,\frac{\nabla\rho_a}{\rho_a}]\rho\|^2+\int_0^T\|[Z^\alpha,(\rho_a+\rho)\nabla]\phi\|^2
+\int_0^T\int_\Omega[Z^\alpha,\nabla]\rho Z^\alpha u\nonumber\\
&+\int_0^T\int_\Omega\mu\epsilon^2[Z^\alpha,\triangle]uZ^\alpha u+\int_0^T\int_\Omega(\mu+\nu)\epsilon^2[Z^\alpha,\nabla\nabla\cdot]uZ^\alpha u\nonumber\\
:=&\int_0^T\|Z^\alpha u\|^2+\sum_{k=1}^8J_{6k}.
\end{align}
For $J_{61}$, a standard commutator estimate yields
\begin{align}
J_{61}&=\int_0^T\|[Z^\alpha,\rho_a+\rho]\partial_tu\|^2\leq \int_0^T(\|Z^3\rho\|^2+\|Z^3u\|^2)+C\epsilon^{2K-6}.\nonumber
\end{align}
Similarly, we have
\begin{align}
J_{62}+J_{63}+J_{64}
&\leq \int_0^T (\|Z^3\rho\|^2+\|Z^3 u\|^2)+C\epsilon^{2K-6}.\nonumber
\end{align}
and
\begin{align}
J_{65}&=\int_0^T\|[Z^\alpha,(\rho_a+\rho)\nabla]\phi\|^2\leq C\int_0^T \|Z^3\rho\|^2+C\epsilon^{2K-6}.\nonumber
\end{align}
For $J_{66}$, since we do not expect to control $\|\nabla\rho\|_{\mathcal{H}^2_{co}}$, we get
\begin{align}
J_{66}&=\int_0^T\int_\Omega[Z^\alpha,\nabla]\rho Z^\alpha u=\int_0^T\int_\Omega\sum_{|\beta|\leq 2} \psi^{\beta,2}\partial_3Z^\beta\rho Z^\alpha u\nonumber\\
&=\int_0^T\int_\Omega\sum_{|\beta|\leq2} \psi^{\beta,2}\psi(x_3)\partial_3Z^\beta\rho \partial_3Z^{|\alpha|-1} u\nonumber\\
&\leq \int_0^T(\|Z^3\rho\|^2+\|\partial_3Z^2 u\|^2)+C\epsilon^{2K-4}\leq \int_0^T\|Z^3\rho\|^2+C\epsilon^{2K-6}.\nonumber
\end{align}
By integrating by parts and the trace theorem, we have
\begin{align}
J_{67}&=\int_0^T\int_\Omega\mu\epsilon^2[Z^\alpha,\triangle]uZ^\alpha u=\mu\epsilon^2\int_0^T\int_\Omega[Z^\alpha,\partial_{33}]uZ^\alpha u\nonumber\\
&=\mu\epsilon^2\int_0^T\int_\Omega\sum_{0\leq|\beta|\leq2}(\psi^{1,\beta,2} \partial_3Z^\beta u
+\psi^{2,\beta,2} \partial_{33}Z^\beta u)Z^\alpha u\nonumber\\
&\leq \mu\epsilon^2\int_0^T \Big(\sum_{0\leq|\beta|\leq2}\|\partial_3Z^\beta u\|^2+\|Z^\alpha u\|^2\Big)
+\mu\epsilon^2\sum_{0\leq|\beta|\leq2}\int_0^T\int_{\partial\Omega}\psi^{2,\beta,2} \partial_{3}Z^\beta u_yZ^\alpha u_y\nonumber\\
&\;\;+\mu\epsilon^2\sum_{0\leq|\beta|\leq2}\int_0^T\int_{\Omega}\partial_{3}Z^\beta u\partial_3(\psi^{2,\beta,2} Z^\alpha u)\nonumber\\
&\leq C\int_0^T\|Z^3 u\|^2+C\epsilon^{2K-4}+\mu\epsilon^2\sum_{0\leq|\beta|\leq2}\int_0^T\int_{\partial\Omega}\psi^{2,\beta,2}
Z^\beta u_yZ^\alpha u_y\nonumber\\
&\;\;-\mu\epsilon^2\sum_{0\leq|\beta|\leq2}\int_0^T\int_{\partial\Omega}\psi^{2,\beta,2}
[Z^\beta,\partial_3] u_yZ^\alpha u_y+\mu\epsilon^2\sum_{0\leq|\beta|\leq2}\int_0^T\int_{\Omega}\partial_3\psi^{2,\beta,2} \partial_{3}Z^\beta uZ^\alpha u\nonumber\\
&\;\;+\mu\epsilon^2\sum_{0\leq|\beta|\leq2}\int_0^T\int_{\Omega}\psi^{2,\beta,2} \partial_{3}Z^\beta u\partial_3Z^\alpha u\nonumber\\
&\leq C\int_0^T \|Z^3 u\|^2+\frac{\mu}{4}\epsilon^2\int_0^T\|\nabla Z^3 u\|^2+C\epsilon^{2K-6}.\nonumber
\end{align}
The estimate for $J_{68}$ is similar, so we omit it for simplicity.
Collecting the above estimates, we get the estimate for $J_6$ that
\begin{align} \label{tan12}
J_6&\leq \int_0^T(\|Z^3\rho\|^2+\|Z^3u\|^2)+ \frac{\mu}{4}\epsilon^2\int_0^T\|\nabla Z^3 u\|^2+C\epsilon^{2K-6}.
\end{align}
Notice that we have given the estimate of $\|\mathcal{C}_\rho\|^2$ in the estimate of $J_5$. Therefore it remains to estimate
$\int_0^T\|\partial_t\mathcal{C}_\phi\|^2+\frac{1}{\epsilon^2}\|\mathcal{C}_\phi\|^2.$
From the expression of $\mathcal{C}_\phi$, one has
\begin{align}
\int_0^T\|\mathcal{C}_\phi\|^2\leq&\int_0^T\|[Z^\alpha,\epsilon^2\triangle]\phi\|^2
+\int_0^T\|Z^\alpha\big(e^{-\phi_a}(e^{-\phi}-1)\big)+e^{-\phi_a}Z^\alpha\phi e^{-\phi}\|^2\nonumber\\
\leq&\;\epsilon^4\int_0^T\sum^{2}_{|\beta|=0}\|(\psi^{1,\beta,2}Z^\beta\partial_3\phi+\psi^{2,\beta,2}Z^\beta\partial_{33}\phi)\|^2+C\epsilon^{2K-4}\nonumber\\
\leq&\;C\epsilon^4\int_0^T\sum^{2}_{|\beta|=0}\|Z^\beta\partial_{33}\phi\|^2+C\epsilon^{2K-4}.\nonumber
\end{align}
Using the Poisson equation, we have
\begin{align}
\int_0^T\|\mathcal{C}_\phi\|^2&\leq C\epsilon^4\int_0^T\sum^{2}_{|\beta|=0}\|Z^\beta\triangle_{y}\phi\|^2
+C\epsilon^{2K-4}\nonumber\\
&\leq C\epsilon^4\int_0^T\|\nabla Z^3\phi\|^2+C\epsilon^{2K-4}.\nonumber
\end{align}
Now, let's consider $\partial_t\mathcal{C}_\phi$. In fact, we have
\begin{align}
\int_0^T\|\partial_t\mathcal{C}_\phi\|^2&\leq\int_0^T\|\partial_t[Z^\alpha,\epsilon^2\triangle]\phi\|^2\nonumber\\
&\;\;+\int_0^T\|\partial_t\Big(Z^\alpha\big(e^{-\phi_a}(e^{-\phi}-1)\big)+e^{-\phi_a}Z^\alpha\phi e^{-\phi}\Big)\|^2\nonumber\\
&\leq\epsilon^2\int_0^T\sum^{3}_{|\beta|=0}\|(\psi^{1,\beta,2}Z^\beta\partial_3\phi+\psi^{2,\beta,2}Z^\beta\partial_{33}\phi)\|^2\nonumber\\
&\;\;+\int_0^T\|Z^3\phi\|^2+C\epsilon^{2K-4}\nonumber\\
&\leq C\int_0^T(\epsilon^2\|\nabla Z^3\phi\|^2+\|Z^3\phi\|^2)+\epsilon^2\int_0^T\|\partial_{33}\phi\|_{\mathcal{H}^3_{co}}^2\nonumber\\
&\;\;+C\epsilon^{2K-6}.\nonumber
\end{align}
By a standard elliptic estimate, we have
\begin{align}
\epsilon^2\|\nabla^2\phi\|_{\mathcal{H}^3_{co}}\leq \|\rho\|_{\mathcal{H}^3_{co}}+\|\phi\|_{\mathcal{H}^3_{co}}+\|\mathcal{C}_\phi\|+C\epsilon^{K+1}.\nonumber
\end{align}
Thus it holds that
\begin{align}
\int_0^T\|\partial_t\mathcal{C}_\phi\|^2
\leq&\;C\int_0^T(\epsilon^2\|\nabla Z^3\phi\|^2+\|Z^3\phi\|^2+\|Z^3\rho\|^2)+\epsilon^2\int_0^T\|\mathcal{C}_\phi\|+C\epsilon^{2K-6}\nonumber\\
\leq&\;C\int_0^T(\epsilon^2\|\nabla Z^3\phi\|^2+\|Z^3\phi\|^2+\|Z^3\rho\|^2)+C\epsilon^{2K-6}.\nonumber
\end{align}
Finally, combining all the commutators estimates with \eqref{tan8}, we obtain
\begin{align}
\|Z^3\big(\rho,u,\phi,\epsilon\nabla\phi\big)\|^2
&+\frac{\mu}{4}\epsilon^{2}\int_0^T\|\nabla Z^3 u\|^2\nonumber\\
&\leq C\int_0^T\|Z^3\big(\rho,u,\phi,\epsilon\nabla\phi\big)\|^2+\int_0^T\|Zu\|_{H^3}^2\|Z^3\rho\|^2+C\epsilon^{2K-6}.\nonumber
\end{align}
By Gronwall's inequality, we complete the proof of Proposition \ref{tan}.
\end{proof}

\subsection{Normal estimates}
In this section, we shall give the estimates of the normal derivatives of the solution. For notations simplicity, we set $R:=\partial_3\rho$. First, we have the following proposition.
\begin{prop} \label{Nphi1}
Under the assumptions of Proposition \ref{L2}, it holds that
\begin{align}
\|\partial_3^2\phi\|^2+\epsilon^2\|Z\partial_3^2\phi\|^2+\epsilon^4\|Z^2\partial_3^2\phi\|^2\leq C\epsilon^{2K-4}.\nonumber
\end{align}
\end{prop}
\begin{proof}
Using the Poisson equation, we have
\begin{align} \label{nor1}
\epsilon^2\partial_3^2\phi=-\epsilon^2\triangle_y\phi+\rho-e^{-\phi_a}(e^{-\phi}-1)+\epsilon^{K+1}R_\phi.
\end{align}
Taking $L^2$ norm to \eqref{nor1} yields
\begin{align}
\epsilon^4\|\partial_3^2\phi\|^2\leq \epsilon^4\|\triangle_y\phi\|^2+\|\rho\|^2+\|e^{-\phi_a}(e^{-\phi}-1)\|^2+\epsilon^{2K+2}\leq C\epsilon^{2K},\nonumber
\end{align}
by using the tangential estimates in Proposition \ref{L2} and Proposition \ref{tan}.
Next, applying $Z$ to \eqref{nor1} gives
\begin{align}
\epsilon^2Z\partial_3^2\phi=-\epsilon^2Z\triangle_y\phi+Z\rho-Z[e^{-\phi_a}(e^{-\phi}-1)]+\epsilon^{K+1}ZR_\phi.\nonumber
\end{align}
Taking $L^2$ norm and using the tangential estimates we have established yields
\begin{align}
\epsilon^4\|Z\partial_3^2\phi\|^2&\leq \epsilon^4\|Z\triangle_y\phi\|+\|Z\rho\|+\|Z[e^{-\phi_a}(e^{-\phi}-1)]\|^2+\|\epsilon^{K+1}ZR_\phi\|^2\nonumber\\
&\leq C\epsilon^{2K-2}.\nonumber
\end{align}
The estimates of $Z^2\partial_3^2\phi$ can be given in a similar fashion, so we omit the details for simplicity. Therefore, we complete the proof of Proposition \ref{Nphi1}.
\end{proof}
Next, we prove the estimates of $R$.
\begin{prop} \label{NR}
Under the assumptions of Proposition \ref{L2}, it holds that
\begin{align}
\epsilon^2&\|R\|^2+\int_0^T\|R\|^2\leq C\epsilon^{2K-2}.\nonumber
\end{align}
\end{prop}
\begin{proof}
Rewrite the first equation of \eqref{3.1} into
\begin{equation}   \label{nor2}       
\partial_t\rho+(u_a+u)\cdot\nabla\rho+(\rho_a+\rho)\nabla_y\cdot u_y+(\rho_a+\rho)\partial_3 u_3+u\cdot\nabla\rho_a+\rho\nabla\cdot u_a=\epsilon^KR_\rho.
\end{equation}
Taking derivative with respect to $x_3$ to the above equation, we get
\begin{align}   \label{nor3}       
\partial_t\partial_3\rho+(u_a+u)\cdot\nabla\partial_3\rho&+\partial_3(u_a+u)_3\partial_3\rho
\nonumber\\
&+\partial_3\rho\nabla\cdot (u_a+ u)+(\rho_a+\rho)\partial_{33} u_3
=M_1,
\end{align}
where $M_1$ is given by
\begin{align}  \label{nor4}
M_1=&-\partial_3(u_a+u)_y\cdot\nabla_y\rho-(\rho_a+\rho)\nabla_y\cdot \partial_3u_y-\partial_3(u\cdot\nabla\rho_a)\nonumber\\
&-\partial_3\rho_a\nabla\cdot u-\rho\nabla\cdot \partial_3u_a+\epsilon^K\partial_3R_\rho.
\end{align}
Recalling that the equation satisfied by $u_3$ in \eqref{3.1} is
\begin{align}   \label{nor5}       
(\rho_a+\rho)\big(\partial_t u_3&+(u_a+u)\cdot\nabla u_3+u\cdot\nabla u_a^3\big)+\big(\partial_3 \rho-\frac{\partial_3\rho_a}{\rho_a}\rho\big)\nonumber\\
=&(\rho_a+\rho)\partial_3\phi+\mu\epsilon^2\triangle_y u_3+(2\mu+\nu)\epsilon^2\partial_{33} u_3\nonumber\\
&+(\mu+\nu)\epsilon^2\partial_3\nabla_y\cdot u_y+H(\rho, \rho_a, u_a)_3+\epsilon^KR_u.
\end{align}
Taking $\eqref{nor3}\times (2\mu+\nu)\epsilon^2+\eqref{nor5}\times (\rho_a+\rho)$ gives
\begin{align}   \label{nor6}
(2\mu+\nu)\epsilon^2\Big(\partial_tR&+(u_a+u)\cdot\nabla R+R\partial_3(u_a+u)_3
+R\nabla\cdot (u_a+u)\Big)\nonumber\\
&+(\rho_a+\rho)R=(2\mu+\nu)\epsilon^2M_1+M_2+M_3,
\end{align}
where
\begin{align}
M_2=&(\rho_a+\rho)^2\big(\partial_t u_3+(u_a+u)\cdot\nabla u_3+u\cdot\nabla u_a^3\big)-\frac{\partial_3\rho_a}{\rho_a}(\rho_a+\rho)\rho,\nonumber\\
M_3=&(\rho_a+\rho)^2\partial_3\phi+\mu\epsilon^2(\rho_a+\rho)\triangle_y u_3+(\mu+\nu)\epsilon^2(\rho_a+\rho)\partial_3\nabla_y\cdot u_y\nonumber\\
&+(\rho_a+\rho)(H(\rho, \rho_a, u_a)_3+\epsilon^KR_u).\nonumber
\end{align}
Multiplying \eqref{nor6} by $R$ and integrating over $[0,T]\times \Omega$, we obtain
\begin{align}
(2\mu+\nu)\epsilon^2&\Big(\frac{1}{2}\int_\Omega|R|^2-\frac{1}{2}\int_\Omega|R(0)|^2
+\frac{1}{2}\int_0^T\int_\Omega\nabla\cdot(u_a+u)|R|^2\Big)\nonumber\\
&+(2\mu+\nu)\epsilon^2\int_0^T\int_\Omega\partial_3(u_a+u)_3
|R|^2+\int_0^T\int_\Omega(\rho_a+\rho)|R|^2\nonumber\\
&=\int_0^T\int_\Omega\Big((2\mu+\nu)\epsilon^2M_1+M_2+M_3\Big)R.\nonumber
\end{align}
This gives
\begin{align}   \label{nor7}
\epsilon^2\|R\|^2&+\int_0^T\|R\|^2\leq\int_0^T\|(2\mu+\nu)\epsilon^2M_1+M_2+M_3\|^2+C\epsilon^{2K-2}.
\end{align}
It remains to estimate the first term on the right-hand side of \eqref{nor7}. In fact, from the tangential estimates we have established in Section 4 and Section 5, we have
\begin{align} \label{nor8}
\int_0^T\|\big((2\mu+\nu)\epsilon^2M_1+M_2+M_3\big)\|^2\leq C\epsilon^{2K-2}.
\end{align}
Combining the estimates \eqref{nor7} and \eqref{nor8}, we complete the proof of Proposition \ref{NR}.
\end{proof}

\begin{prop} \label{NZR}
Under the assumptions of Proposition \ref{L2}, it holds that
\begin{align}
\epsilon^2&\|ZR\|^2+\int_0^T\|ZR\|^2\leq C\epsilon^{2K-4}.\nonumber
\end{align}
\end{prop}
\begin{proof}
Applying $Z$ to the equation \eqref{nor6} yields
\begin{align}   \label{nor9}
(2\mu+\nu)\epsilon^2\Big(&\partial_tZ R+(u_a+u)\cdot\nabla Z R+ZR\partial_3(u_a+u)_3+Z R\nabla\cdot (u+ u_a)
\Big)\nonumber\\
&+(\rho_a+\rho)Z R=(2\mu+\nu)\epsilon^2Z M_1+Z M_2+Z M_3+\mathcal{C},
\end{align}
where the commutator $\mathcal{C}$ is
\begin{align}   \label{nor10}
\mathcal{C}&=-(2\mu+\nu)\epsilon^2\Big([Z,(u_a+u)\cdot\nabla]R+[Z,\partial_3(u_a+u)_3]R\Big)\nonumber\\
&-(2\mu+\nu)\epsilon^2[Z,\nabla\cdot (u+ u_a)]R-[Z,(\rho_a+\rho)] R.
\end{align}
Similar to the proof of Proposition \ref{NR}, a direct energy method yields
\begin{align}   \label{nor11}
\epsilon^2\|Z R\|^2&+\int_0^T\|Z R\|^2\nonumber\\
&\leq C\int_0^T\|\epsilon^2Z M_1+Z M_2+Z M_3\|^2+\int_0^T\|\mathcal{C}\|^2+\epsilon^{2K}.
\end{align}
Using the tangential estimates in Proposition \ref{L2} and \ref{tan}, we can estimate the source term as
\begin{align}
\int_0^T\|\epsilon^2Z M_1+Z M_2+Z M_3\|^2\leq \epsilon^{2K-4}.\nonumber
\end{align}
Now we give the control of the second term on the right-side of \eqref{nor11}. From the expression of $\mathcal{C}$, we have
\begin{align}
\int_0^T\|\mathcal{C}\|^2=&\; \epsilon^2\int_0^T\|[Z,(u_a+u)\cdot\nabla]R\|^2+\epsilon^2\int_0^T\|[Z,\partial_3(u_a+u)_3]R\|^2\nonumber\\
&+\epsilon^2\int_0^T\|[Z,\nabla\cdot (u+ u_a)]R\|^2+\int_0^T\|[Z,(\rho_a+\rho)] R\|^2\nonumber\\
:=&\sum_{k=1}^4I_k.\nonumber
\end{align}
We need to give the estimates of each term. For $I_1$, we have
\begin{align}
I_1&= \epsilon^2\int_0^T\|[Z,(u_a+u)\cdot]\nabla R\|^2+\epsilon^2\int_0^T\|(u_a+u)\cdot[Z,\nabla]R\|^2\nonumber\\
&\leq \epsilon^2\int_0^T\|Z(u_a+u)\cdot\nabla R\|^2+\epsilon^2\int_0^T\|(u_a+u)_3\cdot \partial_3R\|^2\nonumber\\
&\leq \epsilon^2\int_0^T\|\frac{1}{\psi(x_3)}Zu_a\cdot  Z R\|^2+\epsilon^2\int_0^T\|Zu\|_{L^3}^2\|\nabla R\|_{L^6}^2\nonumber\\
&\;\;+\epsilon^2\int_0^T\|\frac{1}{\psi(x_3)}u^3_a\cdot ZR\|^2+\epsilon^2\int_0^T\|u_3\|_{L^3}^2\|\partial_3R\|_{L^6}^2\nonumber\\
&\leq \epsilon^2\int_0^T\|\nabla Zu_a\|^2_{L^\infty}\|Z R\|^2+\epsilon^2\int_0^T(\|Zu\|+\|\nabla Zu\|)^2\|\nabla^2 R\|^2\nonumber\\
&\;\;+\epsilon^2\int_0^T\|\nabla u^3_a\|^2_{L^\infty}\| ZR\|^2+\epsilon^2\int_0^T(\|u_3\|+\|\nabla u_3\|)^2\|\nabla^2 R\|^2\nonumber\\
&\leq \epsilon^2 C\int_0^T\|Z R\|^2+C\epsilon^{2K-2}.\nonumber
\end{align}
The estimates of $I_2$ and $I_3$ are similar. Actually, we have
\begin{align}
I_2+I_3&\leq \epsilon^2\int_0^T\|Z\partial_3(u_a+u)_3R\|^2+\epsilon^2\int_0^T\|Z\nabla\cdot (u+ u_a)R\|^2\nonumber\\
&\leq \epsilon^2\int_0^T\|Z\partial_3u_a\|^2_{L^\infty}\|R\|^2+\epsilon^2\int_0^T\|Z\partial_3u_3\|^2\|R\|^2_{L^\infty}\nonumber\\
&\;\;+\epsilon^2\int_0^T\|Z\nabla\cdot u\|^2\|R\|_{L^\infty}^2+\epsilon^2\int_0^T\|Z\nabla\cdot u_a\|_{L^\infty}^2\|R\|^2\nonumber\\
&\leq C\epsilon^{2K-2}.\nonumber
\end{align}
While for $I_4$, we have
\begin{align}
I_4&=\int_0^T\|Z(\rho_a+\rho) R\|^2
\leq \int_0^T\|Z\rho_a R\|^2+\int_0^T\|Z\rho R\|^2\nonumber\\
&\leq \int_0^T\|Z\rho_a\|^2_{L^\infty} \|R\|^2+\int_0^T\|Z\rho\|^2 \|R\|_{L^\infty}^2\nonumber\\
&\leq C\epsilon^{2K-2}.\nonumber
\end{align}
Substituting the above estimates into \eqref{nor11} gives
\begin{align}   \label{nor12}
(2\mu+\nu)\epsilon^2&\|Z R\|^2+\int_0^T\|Z R\|^2\leq\epsilon^2 C\int_0^T\|Z R\|^2+C\epsilon^{2K-4}.
\end{align}
By the smallness of $\epsilon$, we finish the proof of Proposition \ref{NZR}.
\end{proof}

\begin{prop} \label{N3u}
Under the assumptions of Proposition \ref{L2}, it holds that
\begin{align}
\|\partial_3u\|^2+\epsilon^2\int_0^T\|\partial_3^2u\|^2\leq C\epsilon^{2K-4}.\nonumber
\end{align}
\end{prop}
\begin{proof}
Using the first equation of \eqref{3.1}, we have
\begin{align} \label{nor13}
-\partial_3 u_3&=\frac{1}{\rho_a+\rho}[\partial_t\rho+(u_a+u)\cdot\nabla\rho]+\nabla_y\cdot u_y\nonumber\\
&\;\;+\frac{1}{\rho_a+\rho}[u\cdot\nabla\rho_a+\rho\nabla\cdot u_a-\epsilon^KR_\rho],
\end{align}
which yields
\begin{align}
\|\partial_3 u_3\|&\leq C(\|\partial_t\rho\|+\|\nabla\rho\|+\|\nabla_y\cdot u_y\|+\|u\cdot\nabla\rho_a\|+\|\rho\nabla\cdot u_a\|+\epsilon^K)\nonumber\\
&\leq C\epsilon^{K-2},\nonumber
\end{align}
where we have used the estimates in Proposition \ref{NR} and in previous two sections.
From the equation of $u_3$, we have
\begin{align}  \label{nor14}       
(2\mu+\nu)\epsilon^2\partial_{33} u_3=&\;(\rho_a+\rho)\big(\partial_t u_3+(u_a+u)\cdot\nabla u_3+u\cdot\nabla u_a^3\big)\nonumber\\
&+\big(\partial_3 \rho-\frac{\partial_3\rho_a}{\rho_a}\rho\big)-(\rho_a+\rho)\partial_3\phi+\mu\epsilon^2\triangle_y u_3\nonumber\\
&+(\mu+\nu)\epsilon^2\partial_3\nabla_y\cdot u_y+H(\rho, \rho_a, u_a)_3+\epsilon^KR_u.
\end{align}
Taking $L^2L^2$ norm to both side and utilizing Proposition \ref{NR} and the tangential estimates of Proposition \ref{L2} and \ref{tan} gives
\begin{align}         
\epsilon^4\int_0^T\|\partial_{33} u_3\|^2\leq& \;C\int_0^T\big(\|\partial_t u_3\|^2+\|\nabla u_3\|^2+\|u\|^2\big)\nonumber\\
&+\int_0^T\big(\|\partial_3 \rho\|^2+\frac{1}{\epsilon^2}\|\rho\|^2+\|\partial_3\phi\|^2+\epsilon^4\|\triangle_y u_3\|^2\big)\nonumber\\
&+\int_0^T\big(\epsilon^4\|\partial_3\nabla_y\cdot u_y\|^2+\|H(\rho, \rho_a, u_a)_3\|^2+\epsilon^K\|R_u\|^2\big)\nonumber\\
\leq& \;C\epsilon^{2K-2}.\nonumber
\end{align}
Next, we are going to control $\|\partial_3u_y\|$ which can be achieved by estimating the vorticity
$$\omega=
\left(
\begin{array}{cccccccc}
\partial_2u_3-\partial_3u_2\\
\partial_3u_1-\partial_1u_3\\
\partial_1u_2-\partial_2u_1\\
\end{array}
\right).$$
Actually, from the above expression, we find that the estimate of $\|\partial_3u_y\|$ is equivalent to the estimate of $\|\omega_y\|$.
Applying $\nabla\times$ to the second equation of \eqref{3.1} gives
\begin{align}   \label{curl equation}       
(\rho_a+\rho)&\big(\partial_t \omega+(u_a+u)\cdot\nabla \omega\big)=\mu\epsilon^2\triangle w+F,
\end{align}
with the source term $F$ given by
\begin{align}        
F=&-\nabla(\rho_a+\rho)\times\partial_t u+\omega\cdot\nabla[(\rho_a+\rho)(u_a+u)]-\nabla\times((\rho_a+\rho)u\cdot\nabla u_a)\nonumber\\
&+\nabla\frac{\rho}{\rho_a}\times\nabla\rho_a+\nabla(\rho_a+\rho)\times\nabla\phi+\nabla\times H(\rho, \rho_a, u_a)+\epsilon^K\nabla\times R_u.\nonumber
\end{align}
From the boundary condition of $u$, we can determine the boundary condition of $\omega_y$ as
$$\omega_y\mid_{x_3=0}=
\left(
\begin{array}{cccccccc}
\partial_2u_3-\partial_3u_2\\
\partial_3u_1-\partial_1u_3\\
\end{array}
\right)\mid_{x_3=0}=
\left(
\begin{array}{cccccccc}
-u_2\\
u_1\\
\end{array}
\right):=u_y^\perp.$$
This motivates us to introduce $\eta=\omega_y-u_y^\perp$, since it satisfies
\begin{align}         
\eta\mid_{x_3=0}=0.\nonumber
\end{align}
Due to the estimates of $u_y^\perp$, it sufficient to give the control of $\eta$.
Using \eqref{curl equation} and the equation of $u_y$, we have
\begin{align}   \label{eta}      
(\rho_a+\rho)&\big(\partial_t \eta+(u_a+u)\cdot\nabla \eta\big)=\mu\epsilon^2\triangle \eta+F_y+G,
\end{align}
where
\begin{align}         
G&=(\rho_a+\rho)u\cdot\nabla(u_a^y)^\perp+(\nabla_y^\perp\rho-\frac{\rho}{\rho_a}\nabla_y^\perp\rho_a)-(\rho_a+\rho)\nabla_y^\perp\phi\nonumber\\
&-(\mu+\nu)\epsilon^2\nabla_y^\perp\nabla\cdot u-H_y(\rho)^\perp-\epsilon^{K}R_u^\perp.\nonumber
\end{align}
Now, multiplying \eqref{eta} by $\eta$, integrating over $[0,T]\times\Omega$ yields
\begin{align}        
\int_\Omega&(\rho_a+\rho)\eta^2+\mu\epsilon^2\int_0^T\int_\Omega|\nabla \eta|^2=\int_0^T\int_\Omega\partial_t(\rho_a+\rho)\eta^2\nonumber\\
&+\int_0^T\int_\Omega\nabla\cdot((\rho_a+\rho)(u_a+u))\eta^2+\int_0^T\int_\Omega (F_y+G)\cdot\eta.\nonumber
\end{align}
By the tangential estimates of Proposition \ref{L2} and \ref{tan} and Proposition \ref{NR}, the source term can be controlled as
\begin{align}        
\int_0^T\|(F_y+G)\|^2\leq C\epsilon^{2K-4}.\nonumber
\end{align}
Therefore, we get
\begin{align}        
\|\eta\|^2+\epsilon^2\int_0^T\|\nabla \eta\|^2\leq C\epsilon^{2K-4}.\nonumber
\end{align}
By the definition of $\eta$, we find
\begin{align}
\|\partial_3u_y\|^2\leq \|\omega_y\|^2+\|\nabla_yu_3\|^3\leq \|\eta\|^2+\|u_y^\perp\|^2+ \|\nabla_yu_3\|^3\leq C\epsilon^{2K-4},\nonumber
\end{align}
and
\begin{align}
\epsilon^2\int_0^T\|\partial_3^2u_y\|^2
&\leq \epsilon^2\int_0^T\|\partial_3\omega_y\|^2+\epsilon^2\int_0^T\|\partial_3\nabla_yu_3\|^2\nonumber\\
&\leq \epsilon^2\int_0^T\|\partial_3\eta\|^2+\epsilon^2\int_0^T\|u_y^\perp\|^2+\epsilon^2\int_0^T\|\partial_3\nabla_yu_3\|^2\nonumber\\
&\leq C\epsilon^{2K-4}.\nonumber
\end{align}
The proof of Proposition \ref{N3u} is completed.
\end{proof}

\begin{prop} \label{NZ3u}
Under the assumptions of Proposition \ref{L2}, it holds that
\begin{align}
\|Z\partial_3u\|^2+\epsilon^2\int_0^T\|Z\partial_3^2u\|^2\leq C\epsilon^{2K-6}.\nonumber
\end{align}
\end{prop}
\begin{proof}
Using the first equation of \eqref{3.2}, we have
\begin{align}
-Z\partial_3 u_3=&\;Z\Big\{\frac{1}{\rho_a+\rho}[\partial_t\rho+(u_a+u)\cdot\nabla\rho]+\nabla_y\cdot u_y\nonumber\\
&+\frac{1}{\rho_a+\rho}[u\cdot\nabla\rho_a+\rho\nabla\cdot u_a-\epsilon^KR_\rho]\Big\}.\nonumber
\end{align}
So, from the tangential estimates and Proposition \ref{NZR}, we get
\begin{align}
\|Z\partial_3 u_3\|^2&\leq C\epsilon^{2K-6}.\nonumber
\end{align}
Applying $Z$ to \eqref{nor14} yields
\begin{align}         
(2\mu+\nu)\epsilon^2Z\partial_{33} u_3=&\;Z[(\rho_a+\rho)\big(\partial_t u_3+(u_a+u)\cdot\nabla u_3+u\cdot\nabla u_a^3\big)]\nonumber\\
&+Z\big(\partial_3 \rho-\frac{\partial_3\rho_a}{\rho_a}\rho\big)-Z[(\rho_a+\rho)\partial_3\phi]+\mu\epsilon^2Z\triangle_y u_3\nonumber\\
&+(\mu+\nu)\epsilon^2Z\partial_3\nabla_y\cdot u_y+ZH(\rho, \rho_a, u_a)_3+\epsilon^KZR_u.\nonumber
\end{align}
Taking $L^2L^2$ norm to both side and using the estimates in Proposition \ref{L2}, \ref{tan} and \ref{NZR} gives
\begin{align}         
\epsilon^4\int_0^T\|Z\partial_{33} u_3\|^2&\leq C\epsilon^{2K-4}.\nonumber
\end{align}
Next, applying $Z$ to \eqref{eta} gives
\begin{align}   \label{Zeta}      
(\rho_a+\rho)&\big(\partial_t Z\eta+(u_a+u)\cdot\nabla Z\eta\big)=\mu\epsilon^2\triangle Z\eta+ZF_y+ZG+\mathcal{C}^1_{\eta},
\end{align}
where
\begin{align}         
\mathcal{C}^1_{\eta}=-[Z,\rho_a+\rho]\partial_t\eta-[Z,(\rho_a+\rho)(u_a+u)\cdot\nabla]\eta+\mu\epsilon^2[Z,\triangle]\eta.\nonumber
\end{align}
A direct energy estimate to equation \eqref{Zeta} yields
\begin{align}      \label{nor15}  
\|Z\eta\|^2+\mu\epsilon^2\int_0^T\|\nabla Z\eta\|^2\leq&\; C\int_0^T\|Z\eta\|^2+\int_0^T\|ZF_y\|^2+\int_0^T\|ZG\|^2\nonumber\\
&+\int_0^T\int_\Omega Z\eta\cdot\mathcal{C}^1_{\eta}+C\epsilon^{2K-6}.
\end{align}
Using the tangential estimates of Proposition \ref{L2} and \ref{tan} and Proposition \ref{NR} and \ref{NZR}, we have
\begin{align}        
\int_0^T\|ZF_y\|^2+\int_0^T\|ZG\|^2\leq C\epsilon^{2K-6}.\nonumber
\end{align}
While for the commutator, one has
\begin{align}         
\int_0^TZ\eta\cdot\mathcal{C}^1_{\eta}&\leq \int_0^T\|Z(\rho_a+\rho)\partial_t\eta\|^2+\|Z[(\rho_a+\rho)(u_a+u)]\cdot\nabla\eta\|^2\nonumber\\
&\;\;+\mu\epsilon^2\int_0^T\int_\Omega Z\eta\cdot(\psi_1\partial_3+\psi_2\partial_{33})\eta+C\int_0^T\|Z\eta\|^2\nonumber\\
&\leq C\int_0^T\|Z\eta\|^2+C\epsilon^2\int_0^T \|\partial_3\eta\|^2+\kappa\mu\epsilon^2\int_0^T\|\partial_3Z\eta\|^2+C\epsilon^{2K-6}\nonumber\\
&\leq \kappa\mu\epsilon^2\int_0^T\|\partial_3Z\eta\|^2+C\int_0^T\|Z\eta\|^2+C\epsilon^{2K-6},\nonumber
\end{align}
with $\kappa$ being sufficiently small,  by integrating by parts and H\"{o}lder inequality. Putting the above two estimates into \eqref{nor15} and using the smallness of $\kappa$, we complete the proof of Proposition \ref{NZ3u}.
\end{proof}

\begin{prop} \label{Nphi2}
Under the assumptions of Proposition \ref{L2}, it holds that
\begin{align}
\|\partial_3^3\phi\|^2+\epsilon^2\|Z\partial_3^3\phi\|^2\leq C\epsilon^{2K-8}.\nonumber
\end{align}
\end{prop}
\begin{proof}
Applying $\partial_3$ to the Poisson equation, we have
\begin{align}
\epsilon^2\partial_3^3\phi=-\epsilon^2\partial_3\triangle_y\phi+\partial_3\rho-\partial_3[e^{-\phi_a}(e^{-\phi}-1)]+\epsilon^{K+1}\partial_3R_\phi.\nonumber
\end{align}
Applying $Z$ to the above equation yields
\begin{align}
\epsilon^2Z\partial_3^3\phi=-\epsilon^2Z\partial_3\triangle_y\phi+Z\partial_3\rho-Z\partial_3[e^{-\phi_a}(e^{-\phi}-1)]+\epsilon^{K+1}Z\partial_3R_\phi.\nonumber
\end{align}
Then, by taking $L^2$ norm to both side of the above two equations and using the estimates we just established in \ref{NZR}, we can get the desired estimates. Hence, we complete the proof of Proposition \ref{Nphi2}.
\end{proof}

\begin{prop} \label{NZ2R}
Under the assumptions of Proposition \ref{L2}, it holds that
\begin{align}
\epsilon^2&\|Z^2R\|^2+\int_0^T\|Z^2R\|^2\leq C\epsilon^{2K-6}.\nonumber
\end{align}
\end{prop}
\begin{proof}
For $|\alpha|=2$, applying $Z^\alpha$ to the equation \eqref{nor6} yields
\begin{align}   \label{nor16}
(2\mu+\nu)\epsilon^2\Big(\partial_tZ^\alpha R&+(u_a+u)\cdot\nabla Z^\alpha R+\partial_3(u_a+u)_3Z^\alpha R
\Big)\nonumber\\
&+(2\mu+\nu)\epsilon^2Z^\alpha R\nabla\cdot (u+ u_a)+(\rho_a+\rho)Z^\alpha R\nonumber\\
=&(2\mu+\nu)\epsilon^2Z^\alpha M_1+Z^\alpha M_2+Z^\alpha M_3+\mathcal{C},
\end{align}
where the commutator $\mathcal{C}$ is
\begin{align}
\mathcal{C}&=-(2\mu+\nu)\epsilon^2\Big([Z^\alpha,(u_a+u)\cdot\nabla]R+[Z^\alpha,\partial_3(u_a+u)_3]R\Big)\nonumber\\
&-(2\mu+\nu)\epsilon^2[Z^\alpha,\nabla\cdot (u+ u_a)]R-[Z^\alpha,(\rho_a+\rho)] R.\nonumber
\end{align}
We perform energy estimate to \eqref{nor16} to obtain that
\begin{align}   \label{nor17}
\epsilon^2\|Z^\alpha R\|^2+\int_0^T\|Z^\alpha R\|^2\leq&\int_0^T\|(2\mu+\nu)\epsilon^2Z^\alpha M_1+Z^\alpha M_2+Z^\alpha M_3\|^2\nonumber\\
&+\int_0^T\int_\Omega \mathcal{C}^2+C\epsilon^{2K}.
\end{align}
Let's deal with the first term on the right-hand side of \eqref{nor17}. From the tangential estimates in Section 4 and Section 5, after a complicated but straightforward computation, we get
\begin{align}
\int_0^T&\|\epsilon^2Z^\alpha M_1\|^2+\|Z^\alpha M_2\|^2+\|Z^\alpha M_3\|^2\leq  C\epsilon^{2K-6}.\nonumber
\end{align}
We next turn to the estimate of the commutator $\mathcal{C}$. First,
\begin{align}
\int_0^T\|\mathcal{C}\|^2&\leq\epsilon^2\int_0^T\|[Z^\alpha,(u_a+u)\cdot\nabla]R\|^2+\epsilon^2\int_0^T\|[Z^\alpha,\partial_3(u_a+u)_3]R\|^2\nonumber\\
&\;\;+\epsilon^2\int_0^T\|[Z^\alpha,\nabla\cdot (u+ u_a)]R\|^2+\int_0^T\|[Z^\alpha,(\rho_a+\rho)] R\|^2\nonumber\\
&:=\sum_{k=1}^4I_k.\nonumber
\end{align}
For $I_1$, one has
\begin{align}
I_1&\leq\epsilon^2\int_0^T\|[Z^2,(u_a+u)\cdot]\nabla R\|^2
+\epsilon^2\int_0^T\|(u_a+u)\cdot[Z^2,\nabla]R\|^2\nonumber\\
&\leq \epsilon^2\int_0^T\|Z^2(u_a+u)\cdot\nabla R\|^2+\epsilon^2\int_0^T\|Z(u_a+u)\cdot Z\nabla R\|^2\nonumber\\
&\;\;+\epsilon^2\int_0^T\|(u_a+u)\cdot[Z^2,\nabla]R\|^2.\nonumber
\end{align}
Thus we have
\begin{align}
I_1
\leq &\;\epsilon^2\int_0^T\|Z^2u_a\cdot\nabla R\|^2+\epsilon^2\int_0^T\|Z^2u\cdot\nabla R\|^2\nonumber\\
&+\epsilon^2\int_0^T\|Zu_a\cdot Z\nabla R\|^2+\epsilon^2\int_0^T\|Zu\cdot Z\nabla R\|^2\nonumber\\
&+\epsilon^2\int_0^T\|(u_a+u)_3\cdot\partial_3 R\|^2+\epsilon^2\int_0^T\|(u_a+u)_3\cdot\partial_3 ZR\|^2\nonumber\\
\leq&\; \epsilon^2\int_0^T\|\frac{1}{\psi(z)}Z^2u_a\cdot ZR\|^2+\epsilon^2\int_0^T\|Z^2u\|^2_{L^3}\|\nabla R\|_{L^6}^2\nonumber\\
&+\epsilon^2\int_0^T\|\frac{1}{\psi(x_3)}Zu_a\cdot Z^2 R\|^2+\epsilon^2\int_0^T\|\frac{1}{\psi(x_3)}ZuZ^2 R\|^2\nonumber\\
&+\epsilon^2\int_0^T\|\frac{1}{\psi(x_3)}u_a\cdot ZR\|^2+\epsilon^2\int_0^T\|u\|_{L^3}^2\|\partial_3 R\|_{L^6}^2\nonumber\\
&+\epsilon^2\int_0^T\|\frac{1}{\psi(x_3)}u_a\cdot Z^2R\|^2+\epsilon^2\int_0^T\|u_3\partial_3Z R\|^2\nonumber\\
\leq&\; \epsilon^2\int_0^T\|\nabla Z^2u_a\|^2_{L^\infty}\|ZR\|^2+\epsilon^2\int_0^T(\|Z^2u\|^2+\|\nabla Z^2u\|^2)\|\nabla^2 R\|^2\nonumber\\
&+\epsilon^2\int_0^T\|\nabla Zu_a\|_{L^\infty}^2\|Z^2 R\|^2+\epsilon^2\int_0^T\|\nabla Zu\|_{L^\infty}^2\| Z^2 R\|^2\nonumber\\
&+\epsilon^2\int_0^T\|\nabla u_a\|_{L^\infty}^2\| ZR\|^2+\epsilon^2\int_0^T(\|u\|^2+\|\nabla u\|^2)\|\nabla^2 R\|^2\nonumber\\
&+\epsilon^2\int_0^T\|\nabla u_a\|_{L^\infty}^2\| Z^2R\|^2+\epsilon^2\int_0^T\|u\|_{L^\infty}^2\| \partial_3ZR\|^2\nonumber\\
\leq&\; C_a\epsilon^2\int_0^T\| Z^2R\|^2+ C\epsilon^2\int_0^T\|Zu\|^2_{H^3}\| Z^2R\|^2+C\epsilon^{2K-4}.\nonumber
\end{align}
For the estimates of $I_2$ and $I_3$, we obtain that
\begin{align}
I_2+I_3\leq &\;\epsilon^2\int_0^T\|Z^2\partial_3(u_a+u)_3R\|^2+\|Z\partial_3(u_a+u)_3ZR\|^2\nonumber\\
&+\epsilon^2\int_0^T\|Z^2\nabla\cdot (u+ u_a)R\|^2+\|Z\nabla\cdot (u+ u_a)ZR\|^2\nonumber\\
\leq&\; \epsilon^2\int_0^T\|Z^2\partial_3u_a\|_{L^\infty}^2\|R\|^2+\|Z^2\partial_3u_3\|^2\|R\|_{L^\infty}^2\nonumber\\
&+\epsilon^2\int_0^T\|Z\partial_3u_a\|_{L^\infty}^2\|ZR\|^2+\|Z\partial_3u_3\|_{L^6}^2\|ZR\|_{L^3}^2\nonumber\\
&+\epsilon^2\int_0^T\|Z^2\nabla\cdot u_a\|_{L^\infty}^2\|R\|^2+\|Z^2\nabla\cdot u\|^2\|R\|_{L^\infty}^2\nonumber\\
&+\epsilon^2\int_0^T\|Z\nabla\cdot  u_a\|^2_{L^\infty}\|ZR\|^2+\|Z\nabla\cdot u\|^2_{L^6}\|ZR\|_{L^3}^2\nonumber\\
\leq&\; C\epsilon^{2K-6}+\epsilon^2\int_0^T\|Z\partial_{33} u\|^2\leq C\epsilon^{2K-6}.\nonumber
\end{align}
Finally, for $I_4$, we have
\begin{align}
I_4=&\int_0^T\|[Z^\alpha,(\rho_a+\rho)] R\|^2=\int_0^T\|Z^2(\rho_a+\rho) R\|^2+\int_0^T\|Z(\rho_a+\rho) ZR\|^2\nonumber\\
&\leq \int_0^T\|Z^2\rho_a\|^2_{L^\infty} \|R\|^2+\int_0^T\|Z^2\rho\|^2 \|R\|_{L^\infty}^2\nonumber\\
&+\int_0^T\|Z\rho_a\|^2_{L^\infty} \|ZR\|^2+\int_0^T\|Z\rho\|^2_{L^\infty} \|ZR\|^2\nonumber\\
\leq& \;C\epsilon^{2K-6}.\nonumber
\end{align}
Thus, from \eqref{nor17} and the above estimates for $I_j$, we find that
\begin{align}
\epsilon^2&\|Z^2 R\|^2+\frac{1}{2}\int_0^T\|Z^2R\|^2\leq C_a\epsilon^2\int_0^T\| Z^2R\|^2+C\epsilon^{2K-6}.\nonumber
\end{align}
The proof of the Proposition \ref{NZ2R} is completed.

\end{proof}

\begin{prop} \label{NZ23u}
Under the assumptions of Proposition \ref{L2}, it holds that
\begin{align}
\|Z^2\partial_3u\|^2+\epsilon^2\int_0^T\|Z^2\partial_3^2u\|^2\leq C\epsilon^{2K-8}.\nonumber
\end{align}
\end{prop}


\begin{proof}
Based on Proposition \ref{NZ2R}, the proof here can be given in a similar fashion as Proposition \ref{NZ3u}.
Indeed, we only need to consider the energy estimate for the following equation which is obtained by applying $Z^2$ to \eqref{eta}
\begin{align}        
(\rho_a+\rho)&\big(\partial_t Z^2\eta+(u_a+u)\cdot\nabla Z^2\eta\big)=\mu\epsilon^2\triangle Z^2\eta+Z^2F_y+Z^2G+\mathcal{C}^2_{\eta},
\end{align}
with $\mathcal{C}^2_{\eta}$ being the commutator.
We omit the details for the sake of simplicity.
\end{proof}
Next, we are going to deal with the estimates of the second order normal derivatives of $\rho$.
\begin{prop} \label{N3R}
Under the assumptions of Proposition \ref{L2}, it holds that
\begin{align}
\epsilon^2&\|\partial_3R\|^2+\int_0^T\|\partial_{3}R\|^2\leq C\epsilon^{2K-4}.\nonumber
\end{align}
\end{prop}

\begin{proof}
Applying $\partial_3$ to equation \eqref{nor6} gives
\begin{align}   \label{nor18}
(2\mu+\nu)\epsilon^2\Big(\partial_t\partial_3R&+(u_a+u)\cdot\nabla \partial_3R+\partial_3(u_a+u)_3\partial_3R
\Big)\nonumber\\
&+(2\mu+\nu)\epsilon^2\partial_3R\nabla\cdot (u+ u_a)+(\rho_a+\rho)\partial_3R\nonumber\\
=&\partial_3[(2\mu+\nu)\epsilon^2M_1+M_2+M_3]+\mathcal{C},
\end{align}
where commutator $\mathcal{C}$ is given by
\begin{align}
\mathcal{C}=&-(2\mu+\nu)\epsilon^2\Big([\partial_3,(u_a+u)\cdot\nabla]R+[\partial_3,\partial_3(u_a+u)_3]R\Big)\nonumber\\
&-(2\mu+\nu)\epsilon^2[\partial_3,\nabla\cdot (u+ u_a)]R-[\partial_3,(\rho_a+\rho)] R.\nonumber
\end{align}
Multiplying \eqref{nor18} by $\partial_3R$ and integrating over $[0,T]\times \Omega$, we find
\begin{align}
(2\mu+\nu)\epsilon^2&\Big(\frac{1}{2}\int_\Omega|\partial_3 R|^2-\frac{1}{2}\int_\Omega|\partial_3 R(0)|^2
+\frac{1}{2}\int_0^T\int_\Omega\nabla\cdot(u_a+u)|\partial_3 R|^2\Big)\nonumber\\
&+(2\mu+\nu)\epsilon^2\int_0^T\int_\Omega\partial_3(u_a+u)_3|\partial_3 R|^2+\int_0^T\int_\Omega(\rho_a+\rho)|\partial_3 R|^2\nonumber\\
=&\int_0^T\int_\Omega\partial_3\Big((2\mu+\nu)\epsilon^2M_1+M_2+M_3\Big)\partial_3 R+\int_0^T\int_\Omega\mathcal{C}\partial_3 R.\nonumber
\end{align}
By Young's inequality, one has
\begin{align}   \label{nor19}
(2\mu+\nu)\epsilon^2&\|\partial_3 R\|^2+\int_0^T\|\partial_3 R\|^2\nonumber\\
&\leq\int_0^T\|\partial_3\Big((2\mu+\nu)\epsilon^2M_1+M_2+M_3\Big)\|^2+\int_0^T\|\mathcal{C}\|^2.
\end{align}
Using the estimates in Proposition \ref{Nphi1}-\ref{NZ23u}, we find
\begin{align}
\int_0^T\|\partial_3\Big((2\mu+\nu)\epsilon^2M_1+M_2+M_3\Big)\|^2\leq C\epsilon^{2K-4}.\nonumber
\end{align}
Now let us consider the commutator estimate
\begin{align}
\int_0^T\|\mathcal{C}\|^2&\leq \epsilon^2\int_0^T\|[\partial_3,(u_a+u)\cdot\nabla]R\|^2+\epsilon^2\int_0^T\|[\partial_3,\partial_3(u_a+u)_3]R\|^2\nonumber\\
&\;\;+\epsilon^2\int_0^T\|[\partial_3,\nabla\cdot (u+ u_a)]R\|^2+\int_0^T\|[\partial_3,(\rho_a+\rho)] R\|^2\nonumber\\
&:=\sum_{k=1}^4I_k.\nonumber
\end{align}
For $I_1$, we have
\begin{align}
I_1&\leq \epsilon^2\int_0^T\|\partial_3(u_a+u)\cdot\nabla R\|^2\nonumber\\
&\leq \epsilon^2\int_0^T\|\partial_3u_a\|_{L^\infty}\|\nabla R\|^2+\epsilon^2\int_0^T\|\partial_3u\|_{L^3}^2\|\nabla R\|_{L^6}^2\nonumber\\
&\leq C\epsilon^2\int_0^T\|\partial_3 R\|^2+C\epsilon^{2K-4}.\nonumber
\end{align}
For $I_2$ and $I_3$, we have
\begin{align}
I_2+I_3&\leq \epsilon^2\int_0^T\|[\partial_3,\partial_3(u_a+u)_3]R\|^2+\epsilon^2\int_0^T[\partial_3,\nabla\cdot (u+ u_a)]R\|^2\nonumber\\
&\leq C\epsilon^{2K-4}.\nonumber
\end{align}
For $I_4$, we have
\begin{align}
I_4&\leq \int_0^T\|\partial_3\rho_a R\|^2+\int_0^T\|\partial_3\rho R\|^2\nonumber\\
&\leq \int_0^T\|\partial_3\rho_a\|^2_{L^\infty} \|R\|^2+\int_0^T\|R\|_{L^\infty}^2\|R\|^2\nonumber\\
&\leq C\epsilon^{2K-4}.\nonumber
\end{align}
Combining the above estimates with \eqref{nor19}, we get
\begin{align}
\epsilon^2&\|\partial_3 R\|^2+\int_0^T\|\partial_3 R\|^2\leq \epsilon^{2K-4}.\nonumber
\end{align}
The proof of Proposition \ref{N3R} is completed.
\end{proof}
Next, we have
\begin{prop} \label{Nphi3}
Under the assumptions of Proposition \ref{L2}, it holds that
\begin{align}
\|\partial_3^4\phi\|^2\leq C\epsilon^{2K-10}.\nonumber
\end{align}
\end{prop}
\begin{proof}
Applying $\partial_{33}$ to the Poisson equation, we have
\begin{align}
\epsilon^2\partial^4_{3}\phi=-\epsilon^2\partial_{33}\triangle_y\phi+\partial_{33}\rho-\partial_{33}[e^{-\phi_a}(e^{-\phi}-1)]+\epsilon^{K+1}\partial_{33}R_\phi.\nonumber
\end{align}
Then, by taking $L^2$ norm to both sides of the above two equations and using previous estimates of Proposition \ref{Nphi1}-\ref{N3R}, we complete the proof.
\end{proof}
\begin{prop} \label{NZ3R}
Under the assumptions of Proposition \ref{L2}, it holds that
\begin{align}
\epsilon^2&\|Z\partial_3R\|^2+\int_0^T\|Z\partial_{3}R\|^2\leq C\epsilon^{2K-6}.\nonumber
\end{align}
\end{prop}
\begin{proof}
Applying $Z$ to system \eqref{nor18} gives
\begin{align}   \label{nor20}
(2\mu+\nu)\epsilon^2\Big(\partial_tZ\partial_3R&+(u_a+u)\cdot\nabla Z\partial_3R+\partial_3(u_a+u)_3Z\partial_3R
\Big)\nonumber\\
&+(2\mu+\nu)\epsilon^2Z\partial_3R\nabla\cdot (u+ u_a)+(\rho_a+\rho)Z\partial_3R\nonumber\\
=&Z\partial_3[(2\mu+\nu)\epsilon^2M_1+M_2+M_3]+Z\mathcal{C}+\mathcal{\tilde{C}},
\end{align}
with $\mathcal{\tilde{C}}$ given by
\begin{align}
\mathcal{\tilde{C}}&=-(2\mu+\nu)\epsilon^2\Big([Z,(u_a+u)\cdot\nabla]\partial_3R+[Z,\partial_3(u_a+u)_3]\partial_3R\Big)\nonumber\\
&-(2\mu+\nu)\epsilon^2[Z,\nabla\cdot (u+ u_a)]\partial_3R-[Z,(\rho_a+\rho)] \partial_3R.\nonumber
\end{align}
Multiplying \eqref{nor20} by $Z\partial_3R$ and integrating over $[0,T]\times \Omega$, we find
\begin{align}
\epsilon^2&\|Z\partial_3 R\|^2+\int_0^T\|Z\partial_3 R\|^2\nonumber\\
&\leq\int_0^T\|Z\partial_3\Big((2\mu+\nu)\epsilon^2M_1+M_2+M_3\Big)\|^2+\int_0^T\|Z\mathcal{C}\|^2+\int_0^T\|\mathcal{\tilde{C}}\|^2.\nonumber
\end{align}
In view of Proposition \ref{Nphi1}-\ref{N3R}, we find that
\begin{align}
\int_0^T\|Z\partial_3\Big((2\mu+\nu)\epsilon^2M_1+M_2+M_3\Big)\|^2\leq C\epsilon^{2K-6},\nonumber
\end{align}
and
\begin{align}
\int_0^T\|Z\mathcal{C}\|^2+\|\mathcal{\tilde{C}}\|^2\leq & \;\epsilon^2\int_0^T\|Z[\partial_3,(u_a+u)\cdot\nabla]R\|^2+\epsilon^2\int_0^T\|Z[\partial_3,\partial_3(u_a+u)_3]R\|^2\nonumber\\
&+\epsilon^2\int_0^T\|Z[\partial_3,\nabla\cdot (u+ u_a)]R\|^2+\int_0^T\|Z[\partial_3,(\rho_a+\rho)] R\|^2\nonumber\\
\leq&\; C\epsilon^2\int_0^T\|Z\partial_3 R\|^2+C\epsilon^{2K-6}.\nonumber
\end{align}
The combination of the above three inequalities yields the result in Proposition \ref{NZ3R}. The proof is proved.
\end{proof}

\begin{prop} \label{N33u}
Under the assumptions of Proposition \ref{L2}, it holds that
\begin{align}
\|\partial^2_3u\|^2+\epsilon^2\int_0^T\|\partial_3^3u\|^2\leq C\epsilon^{2K-6}.\nonumber
\end{align}
\end{prop}
\begin{proof}
First, the following estimate can be given in a very similar fashion as Proposition \ref{NZ3u}
\begin{align}
\|\partial^2_3u_3\|^2+\epsilon^2\int_0^T\|\partial_3^3u_3\|^2\leq C\epsilon^{2K-6}.\nonumber
\end{align}
Next, applying $\partial_3$ to \eqref{eta} gives
\begin{align}   \label{nor21}     
(\rho_a+\rho)&\big(\partial_t \partial_3\eta+(u_a+u)\cdot\nabla \partial_3\eta\big)=\mu\epsilon^2\triangle \partial_3\eta+\partial_3(F_y+G)+\mathcal{C}^3_{\eta},
\end{align}
with
\begin{align}     
\mathcal{C}^3_{\eta}=-[\partial_3,(\rho_a+\rho)]\partial_t \eta-[\partial_3,(\rho_a+\rho)(u_a+u)\cdot\nabla] \eta+\mu\epsilon^2[\partial_3,\triangle]\eta.\nonumber
\end{align}
To use energy method, we need to determine the boundary condition for $\partial_3\eta$. Actually, from the expression of $\eta$, we find
$$\partial_3\eta\mid_{x_3=0}=
\left(
\begin{array}{cccccccc}
\partial_{23}u_3-\partial_{33}u_2+\partial_{3}u_2\\
\partial_{33}u_1-\partial_{13}u_3-\partial_{3}u_1\\
\end{array}
\right)\mid_{x_3=0}=
\left(
\begin{array}{cccccccc}
\partial_{23}u_3\\
-\partial_{13}u_3\\
\end{array}
\right)\mid_{x_3=0}.$$
Thus, from energy estimate to \eqref{nor21}, we have
\begin{align}  \label{nor22}    
\|\partial_3\eta\|^2&+\epsilon^2\int_0^T\|\nabla\partial_3\eta\|^2\leq\mu\epsilon^2 \int_0^T\int_{\partial\Omega}\partial_{33}\eta\cdot\partial_{3}\eta
\nonumber\\
&+\int_0^T\|\partial_3(F_y+G)+\mathcal{C}^3_{\eta}\|^2:=I_1+I_2.
\end{align}
Let us consider the estimate of the boundary term $I_1$ first.
Using \eqref{eta}, we can reformulate $I_1$ into
\begin{align}      
I_1&= \mu\epsilon^2\int_0^T\int_{\partial\Omega}\big[(\rho_a+\rho)\big(\partial_t \eta+(u_a+u)\cdot\nabla \eta\big)-F_y-G\big]\cdot\partial_{3}\eta\nonumber\\
&=\mu\epsilon^2\int_0^T\int_{\partial\Omega}\big[F_y+G\big]\cdot\left(
\begin{array}{cccccccc}
-\partial_{23}u_3\\
\partial_{13}u_3\\
\end{array}
\right).\nonumber
\end{align}
Thus, by the trace theorem, we find that
\begin{align}      
I_1&\leq \mu\epsilon^2\int_0^T|(F_y+G)|^2_{L^2(\partial\Omega)}+\mu\epsilon^2\int_0^T|\partial_{3}\nabla_yu_3|^2_{L^2(\partial\Omega)}\nonumber\\
&\leq \mu\epsilon^2\int_0^T\|\nabla (F_y+G)\|\|(F_y+G)\|+\mu\epsilon^2\int_0^T\|\nabla^2\nabla_yu_3\|\|\nabla \nabla_yu_3\|\nonumber\\
&\leq C\epsilon^{2K-6}.\nonumber
\end{align}
Similarly, we can obtain the following estimates by using the results in Proposition \ref{Nphi1}-\ref{NZ3R},
\begin{align}      
I_2&\leq C\epsilon^{2K-6}. \nonumber
\end{align}
Putting the above estimates into \eqref{nor22} gives
\begin{align}     
\|\partial_3\eta\|^2&+\epsilon^2\int_0^T\|\nabla\partial_3\eta\|^2\leq C\epsilon^{2K-6}.\nonumber
\end{align}
The proof of Proposition \ref{N33u} is completed.
\end{proof}

\begin{prop} \label{NZ33u}
Under the assumptions of Proposition \ref{L2}, it holds that
\begin{align}
\|Z\partial^2_3u\|^2+\epsilon^2\int_0^T\|Z\partial_3^3u\|^2\leq C\epsilon^{2K-8}.\nonumber
\end{align}
\end{prop}
\begin{proof}
The proof can be given similarly to Proposition \ref{N33u}, so we omit it for the sake of simplicity.
\end{proof}

\begin{prop} \label{N33R}
Under the assumptions of Proposition \ref{L2}, it holds that
\begin{align}
\epsilon^2&\|\partial^2_3R\|^2+\int_0^T\|\partial_{3}^2R\|\leq C\epsilon^{2K-6}.\nonumber
\end{align}
\end{prop}

\begin{proof}
Applying $\partial^2_3$ to the system \eqref{nor6} gives
\begin{align}   \label{nor23}
(2\mu+\nu)\epsilon^2\Big(\partial_t\partial^2_3R&+(u_a+u)\cdot\nabla \partial_3R+\partial_3(u_a+u)_3\partial^2_3R
\Big)\nonumber\\
&+(2\mu+\nu)\epsilon^2\partial^2_3R\nabla\cdot (u+ u_a)+(\rho_a+\rho)\partial^2_3R\nonumber\\
&=\partial^2_3[(2\mu+\nu)\epsilon^2M_1+M_2+M_3]+\mathcal{C},
\end{align}
where commutator $\mathcal{C}$ is given by
\begin{align}
\mathcal{C}&=-(2\mu+\nu)\epsilon^2\Big([\partial^2_3,(u_a+u)\cdot\nabla]R+[\partial^2_3,\partial_3(u_a+u)_3]R\Big)\nonumber\\
&-(2\mu+\nu)\epsilon^2[\partial^2_3,\nabla\cdot (u+ u_a)]R-[\partial^2_3,(\rho_a+\rho)] R.\nonumber
\end{align}
Using energy method to \eqref{nor23}, we have
\begin{align}   \label{nor24}
\epsilon^2&\|\partial^2_3 R\|^2+\int_0^T\|\partial^2_3 R\|^2\nonumber\\
&\leq \int_0^T\|\partial^2_3\Big((2\mu+\nu)\epsilon^2M_1+M_2+M_3\Big)\|^2+\int_0^T\|\mathcal{C}\|^2+C\epsilon^{2K-6}.
\end{align}
By the estimates we have established in previous propositions, we get
\begin{align}
\int_0^T\|\partial^2_3\Big((2\mu+\nu)\epsilon^2M_1+M_2+M_3\Big)\|^2\leq \epsilon^{2K-6}.\nonumber
\end{align}
The estimate of the commutator can be divided into
\begin{align}
\int_0^T\|\mathcal{C}\|^2\leq&\; \epsilon^2\int_0^T\|[\partial^2_3,(u_a+u)\cdot\nabla]R\|^2+\epsilon^2\int_0^T\|[\partial^2_3,\partial_3(u_a+u)_3]R\|^2\nonumber\\
&+\epsilon^2\int_0^T\|[\partial^2_3,\nabla\cdot (u+ u_a)]R\|^2+\int_0^T\|[\partial^2_3,(\rho_a+\rho)] R\|^2\nonumber\\
:=&\sum_{i=1}^4I_i.\nonumber
\end{align}
For $I_1$, we have
\begin{align}
I_1&\leq \epsilon^2\int_0^T\|\partial^2_3(u_a+u)\cdot\nabla R\|^2
+\epsilon^2\int_0^T\|\partial_3(u_a+u)\cdot\nabla \partial_3R\|^2\nonumber\\
&\leq \epsilon^2\int_0^T\|\partial^2_3u_a\|^2_{L^\infty}\|\nabla R\|^2+\epsilon^2\int_0^T\|\partial^2_3u\|^2_{L^3}\|\nabla R\|_{L^6}^2\nonumber\\
&\;\;+\epsilon^2\int_0^T\|\partial_3 u_a\|^2_{L^\infty}\|\nabla \partial_3R\|^2+\epsilon^2\int_0^T\|\partial_3 u\|^2_{L^\infty}\|\nabla \partial_3R\|^2\nonumber\\
&\leq C_a\int_0^T\|\nabla R\|^2+\epsilon^2\int_0^T(\|\partial^2_3u\|^2+\|\nabla \partial^2_3u\|^2)\|\nabla^2 R\|^2\nonumber\\
&\;\;+C_a\epsilon^2\int_0^T\|\nabla \partial_3R\|^2+\epsilon^2\int_0^T\|\partial_3 u\|^2_{H^2}\|\nabla^2 R\|^2\nonumber\\
&\leq C_a\epsilon^2\int_0^T\| \partial^2_3R\|^2+C\epsilon^{2K-6}.\nonumber
\end{align}
For $I_2$ and $I_3$, it is straightforward to have that
\begin{align}
I_2+I_3&\leq C\epsilon^{2K-6}.\nonumber
\end{align}
For $I_4$, one has
\begin{align}
I_4&\leq \int_0^T\|\partial^2_3(\rho_a+\rho) R\|^2+\int_0^T\|\partial_3(\rho_a+\rho) \partial_3R\|^2\nonumber\\
&\leq \int_0^T\|\partial^2_3\rho_a\|_{L^\infty}^2 \|R\|^2+2\int_0^T\|R\|_{L^\infty}^2 \|\partial_3R\|^2\nonumber\\
&\;\;+\int_0^T\|\partial_3\rho_a\|^2_{L^\infty} \|\partial_3R\|^2\nonumber\\
&\leq C\epsilon^{2K-6}.\nonumber
\end{align}
From \eqref{nor24} and above estimates, one has
\begin{align}
\epsilon^2&\|\partial^2_3 R\|^2+\int_0^T\|\partial^2_3 R\|^2\leq C_a\epsilon^2\int_0^T\| \partial^2_3R\|^2+C\epsilon^{2K-6}.\nonumber
\end{align}
The proof of Proposition \ref{N33R} is completed.
\end{proof}
Finally, we have
\begin{prop} \label{N333u}
It holds that
\begin{align}
\|\partial^3_{3}u\|^2+\epsilon^2&\int_0^T\|\partial^4_{3}u\|^2\leq C\epsilon^{2K-10}.\nonumber
\end{align}
\end{prop}
\begin{proof}
Using the first equation of \eqref{3.1}, we have
\begin{align}
-\partial^3_3 u_3&=\partial_3^2\Big\{\frac{1}{\rho_a+\rho}[\partial_t\rho+(u_a+u)\cdot\nabla\rho]+\nabla_y\cdot u_y\nonumber\\
&\;\;+\frac{1}{\rho_a+\rho}[u\cdot\nabla\rho_a+\rho\nabla\cdot u_a-\epsilon^KR_\rho]\Big\}.\nonumber
\end{align}
Taking $L^2$ norm yields the desired estimate of $\|\partial^3_3 u_3\|$.
While from the equation of $u_3$, we have
\begin{align}         
(2\mu+\nu)\epsilon^2\partial^4_{3} u_3=&\;\partial^2_{3}\Big\{(\rho_a+\rho)\big(\partial_t u_3+(u_a+u)\cdot\nabla u_3+u\cdot\nabla u_a^3\big)\nonumber\\
&+\big(\partial_3 \rho-\frac{\partial_3\rho_a}{\rho_a}\rho\big)-(\rho_a+\rho)\partial_3\phi+\mu\epsilon^2\triangle_y u_3\nonumber\\
&+(\mu+\nu)\epsilon^2\partial_3\nabla_y\cdot u_y+H(\rho, \rho_a, u_a)_3+\epsilon^KR_u\Big\}.\nonumber
\end{align}
Taking $L^2L^2$ norm to both side gives the estimate of $\int_0^T\|\partial^4_{3}u_3\|^2$.
Next, taking $\partial^2_{3}$ to \eqref{eta} gives
\begin{align}        
(\rho_a+\rho)&\big(\partial_t \partial^2_{3}\eta+(u_a+u)\cdot\nabla \partial^2_{3}\eta\big)=\mu\epsilon^2\triangle \partial^2_{3}\eta+\partial^2_{3}(F_y+G)+\mathcal{C}.\nonumber
\end{align}
Multiplying the above equation by $\partial^2_{3}\eta$ and integrating over $[0,T]\times\Omega$ yield
\begin{align}        
\|\partial^2_{3}\eta\|^2+\mu\epsilon^2\int_0^T\|\nabla \partial^2_{3}\eta\|^2=\int_0^T\int_\Omega\partial^3_{3}\eta\cdot\partial^2_{3}\eta+\int_0^T\int_\Omega \partial^2_{3}(F_y+G+\mathcal{C})\cdot\partial^2_{3}\eta.\nonumber
\end{align}
By estimating the source terms and the boundary terms, one has
\begin{align}        
\|\partial^2_{3}\eta\|^2+\mu\epsilon^2\int_0^T\|\nabla \partial^2_{3}\eta\|^2&\leq C\epsilon^{2K-10}.\nonumber
\end{align}
The proof of Proposition \ref{N333u} is completed.
\end{proof}

\noindent{\bf Proof of the main theorem.}
Combining the estimates in the previous three sections, we complete the proof of Proposition \ref{mainproposition} of the a priori estimates.  From the local well-posedness theory in Section 3.1, we know that
\begin{align}
T^\epsilon=\sup\{T>0,\quad \forall t\in[0,T], \quad \Lambda(t)\leq \epsilon^{2}\}>0,
\end{align}
for
$$\Lambda(t):=\|(\rho,u,\phi,\epsilon\nabla\phi)(t)\|^2_{\mathcal{H}^3}+\epsilon^2\int_0^t\|\nabla u\|^2_{\mathcal{H}^3}d\tau.$$

Since $K>6$ so that $2K-10>2$,   in view of the definition of $T^\epsilon$ and Proposition \ref{mainproposition}, we can find $T>0$ such that $T^\epsilon\geq T$ and the following uniform estimate holds on $[0,T]$,
\begin{align}
\sup_{0\leq t\leq T}\|(\rho,u,\phi,\epsilon\nabla\phi)\|_{H^3(\mathbb{R}^3_+)}\leq C\epsilon^{K-5}.\nonumber
\end{align}
Therefore, we have proved our main theorem.
\section{acknowledgements}
 The research of Ju was supported by NSFC(Grant Nos. 12131007 and 12070144). Luo's research was supported by a GRF grant CityU 11306117 of RGC (Hong Kong).
Xu's research was supported by Natural Science Foundation of China (Grant No. 12001506) and Natural Science Foundation of Shandong Province (Grant No. ZR2020QA014). The work of Ju and Xu was also supported by the ISF-NSFC joint research program (NSFC Grant No. 11761141008).



\end{document}